\documentclass[12pt]{article}

\usepackage{amssymb,amsthm,amsmath}
\usepackage{url}

\makeatletter
\@addtoreset{equation}{section}

\ifx\reset@font\undefined
      \let\reset@font=\relax
\fi

\def\eps{\varepsilon}

\def\E{\mathbb E}
\def\Ex{\mathbb E}
\def\p{\mathbb P}
\def\PP{\mathbb P}
\def\Pr{\mathbb P}
\def\er{\mathbb R}

\newcommand*{\ind}[1]{\mathbf{1}_{\{#1\}}}

\newcommand{\R}{\mathbb{R}}
\newcommand{\N}{\mathbb{N}}

\newcommand{\supp}{{\rm supp\,}}

\newtheorem{lemma}{Lemma}[section]

\newtheorem{theorem}[lemma]{Theorem}
\newtheorem{prop}[lemma]{Proposition}
\newtheorem{cor}[lemma]{Corollary}
\newtheorem{thm}[lemma]{Theorem}

\newtheorem{defi}[lemma]{Definition}

\renewcommand\r{\right}

\newcommand\la{\langle}
\newcommand\ra{\rangle}
\newcommand\lam{\lambda}
\newcommand\Ga{\Gamma}

\title{Tail estimates for norms of sums of log-concave random vectors 
}

\author{
Rados{\l}aw Adamczak 
 \thanks{Research partially supported by MNiSW Grant no. N N201 397437
   and the Foundation for Polish Science.}
 \and  Rafa{\l} Lata{\l}a${}^{*}$
 \and  Alexander E. Litvak
\thanks{Research partially supported by  the 
E.W.R. Steacie Memorial Fellowship.}
\and Alain Pajor
\thanks{Research partially supported by the ANR project
ANR-08-BLAN-0311-01.}
\and  Nicole  Tomczak-Jaegermann
\thanks{This author holds the Canada Research Chair in
  Geometric Analysis.}
}

\newcommand\address{\noindent\leavevmode%
\noindent
Rados{\l}aw  Adamczak, \\
Institute of Mathematics, \\
University of Warsaw, \\
Banacha 2, 02-097 Warszawa, Poland\\
 \texttt{\small
e-mail:  radamcz@mimuw.edu.pl}

\medskip
\noindent
Rafa{\l} Lata{\l}a, \\
Institute of Mathematics, \\
University of Warsaw, \\
Banacha 2, 02-097 Warszawa, Poland\\
and \\
Institute of Mathematics,\\
Polish Academy of Sciences,\\
\'{S}niadeckich 8, 00-956 Warszawa, Poland\\
 \texttt{\small
e-mail:   rlatala@mimuw.edu.pl}

\medskip
\noindent
Alexander E. Litvak, \\
Dept.~of Math.~and Stat.~Sciences,\\
University of Alberta, \\
Edmonton, Alberta, Canada, T6G 2G1.\\
\texttt{\small%
e-mail:   alexandr@math.ualberta.ca}

\medskip
\noindent
Alain  Pajor, \\
Universit\'{e} Paris-Est\\
\'{E}quipe d'Analyse et Math\'{e}matiques Appliqu\'ees, \\
5, boulevard Descartes,
Champs sur Marne,\\
77454 Marne-la-Vall\'{e}e,  Cedex 2, France\\
\texttt{\small%
e-mail: Alain.Pajor@univ-mlv.fr }

\medskip
\noindent
Nicole  Tomczak-Jaegermann, \\
Dept.~of Math.~and Stat.~Sciences,\\
University of Alberta, \\
Edmonton, Alberta, Canada, T6G 2G1.\\
\texttt{\small%
e-mail:    nicole.tomczak@ualberta.ca}
}

\date{}

\begin{document}
\maketitle

\begin{abstract}  
  We establish new tail estimates for order statistics and for the
  Euclidean norms of projections of an isotropic log-concave random
  vector.  More generally, we prove tail estimates for the norms of
  projections of sums of independent log-concave random vectors, and
  uniform versions of these in the form of tail estimates for operator
  norms of matrices and their sub-matrices in the setting of a
  log-concave ensemble. This is used to study a quantity $A_{k,m}$
  that controls uniformly the operator norm of the sub-matrices with
  $k$ rows and $m$ columns of a matrix $A$ with independent isotropic
  log-concave random rows.  We apply our tail estimates of $A_{k,m}$
  to the study of Restricted Isometry Property that plays a major role
  in the Compressive Sensing theory.
\end{abstract}



\section{Introduction} 
\label{intro}

In the recent years
a lot of work was done on the study of the
empirical covariance matrix, and on
understanding related random matrices with independent rows or columns.
In particular, such matrices appear naturally in two important
(and distinct) directions. Namely, \\
-- approximation of covariance matrices of high-dimensional
distributions by empirical  covariance matrices;
and \\
-- the Restricted Isometry  Property of sensing matrices defined
in the Compressive Sensing theory.

\vspace{2ex}

To illustrate, let $n, N$ be integers. 
 For $1 \le m \le N$ by $ U_m = U_m(\R^N)$ we denote the set of
 $m$-sparse vectors 
of norm one, that is, vectors $x \in S^{N-1}$ with at most $m$
non-zero coordinates.  
For any $n \times N$ random matrix $A$, treating $A$ as a linear
operator $A: \R^N \to \R^n$ we define $\delta_m (A) $ by $ \delta_m
(A) = \sup_{x\in {U_m}} \Big||Ax|^2 - \E|A x|^2 \Big|$.  (Here
$|\cdot|$ denotes the Euclidean norm on $\R^n$.)

Now let $X\in\R^N$ be a centered random vector with the covariance
matrix equal to the identity, that is, $\E X\otimes X = Id$; such
vectors are called isotropic.  Consider $n$ independent random
vectors $X_1, \ldots, X_n$ distributed as $X$ and let $A$ be the $n
\times N$ matrix whose rows are $X_1, \ldots, X_n$. Then
\begin{align}
  \label{eq:delta-m-1}
  \delta_m \Big(\frac{A}{\sqrt n}\Big) & = \sup_{x\in U_m} \bigl|
  \frac{1}{n} \left(|Ax|^2 
- \E  |Ax|^2 \right) \bigr| \nonumber\\
& = \sup_{x\in U_m}
\Big|\frac{1}{n}\sum_{i=1}^n \left(|\langle
   X_i, x\rangle|^2 - \E |\langle X_i, x\rangle|^2\right)\Big|.
\end{align}
In the particular case of $m=N$ it is also easy to  check that
\begin{equation}
\label{discrepancy'}   
\delta_N\Big(\frac{A}{\sqrt n}\Big)  
= \Big\|\frac{1}{n}\sum_{i=1}^n \left(X_i
\otimes X_i - \E X \otimes X \right) \Big\|. 
\end{equation}

\vspace{2ex}

We first
discuss   the case  $n \ge N$. In this case we will work
only with the parameter $\delta_N(A/\sqrt n)$.
By the law of large numbers, under some moment hypothesis, the empirical
covariance matrix $\frac{1}{n}\sum_{i=1}^n X_i\otimes X_i$ converges
to  $\E\, X\otimes X= Id$ in the operator norm, as $n \to \infty$.
A natural goal important for  many classes of distributions is to get
quantitative estimates of the rate of this convergence, in other
words,  to estimate the error  term $\delta_N(A/\sqrt n)$
with high probability, as $n \to \infty$.

This question was raised and investigated in \cite{KLS} motivated by
a problem of complexity in computing volume in high dimensions.  In
this setting it was natural to consider uniform measures on convex
bodies, or more generally, log-concave measures (see below for all the
definitions). Partial solutions were given in \cite{B} and \cite{R}
soon after the question was raised, and in the intervening years
further partial solutions were produced. 
A full and optimal answer to the Kannan-Lov\'asz- Simonovits
question  was given in   \cite{jams} and \cite{cras_alpt}.
For recent results on similar questions for other distributions, see
e.g., \cite{V, SV}.

The answer from  \cite{jams} and \cite{cras_alpt} to the K-L-S question 
on the rate of convergence  stated that:
\begin{equation}
  \label{eq:1}
\p \left(
   \sup_{x\in S^{N-1}} \Big|\frac{1}{n}\sum_{i=1}^n \left(|\langle
   X_i, x\rangle|^2 - 1\right)\Big|
  \leq C \sqrt{ \frac{N}{n}} \right) \ge 1-e^{-c \sqrt N},
\end{equation}
where $C$ and $c$  are  absolute positive constants.
The proofs  are
based on an approach initiated by J.~Bourgain \cite{B} where the
following norm of a matrix played a central role.  Let $1\le k \le n$, then
\begin{equation}
  \label{eq:AkN}
  A_{k,N} = \sup_{{J \subset \{1,\ldots,n\}}\atop {|J| = k}}
\sup_{x\in S^{N-1}} \Bigl(\sum_{j\in J} |\langle X_j,
x\rangle|^2\Bigr)^{1/2}
= \sup_{{J \subset \{1,\ldots,n\}}\atop {|J| = k}}
\sup_{x\in S^{N-1}} |P_J Ax|,
\end{equation}
where for ${J \subset \{1,\ldots,n\}}$, $P_J$ denotes the
orthogonal projection on the coordinate subspace $\R^J$ of $\R^n$.

To understand the role of $A_{k,N}$ for estimating $\delta_N(A/\sqrt
n)$, let us explain the standard approach. For each individual $x$
on the sphere, the rate of convergence may be estimated via some
probabilistic concentration inequality. The method consists of a 
discretization of the sphere and then the use of an approximation
argument to complete the proof. This approach works perfectly as
long as the trade-off between complexity and concentration allows
it.

Thus when the random variables $\frac{1}{n}\sum_{i=1}^n |\langle
X_i,x\rangle|^2$ satisfy a good concentration inequality sufficient
to handle uniformly exponentially many points, the method works. This
is the case for instance when the random variables $\langle
X,x\rangle$ are sub-gaussian or bounded, due to Bernstein
inequalities. In the general case, we decompose the function $|\langle
X_i,x\rangle|^2$ as the sum of two terms, the first being its
truncation at the level $B^2$, for some $B>0$. Now let us discuss
the second term in the decomposition of $\sum_{i=1}^n |\langle
X_i,x\rangle|^2$ . Let
$$
 E_B=E_B(x)=\{i\leq n\,:\, |\langle X_i,x\rangle|>B\}.
$$
For simplicity, let us assume that the maximum cardinality of the
sets of the family $\{E_{B}(x):\, x\in S^{N-1}\}$ is a fixed
non-random number $k$, then clearly the second term is controlled by
$$\sum_{i\in E_B} |\langle X_i,x\rangle|^2\leq A_{k,N}^2.$$
In order to estimate $k$, let $x$ such that  $k=|E_B(x)|=|E_B|$, then
$$
  B^2 k=B^2 |E_B|\leq \sum_{i\in E_B} |\langle X_i,x\rangle|^2.
$$
Thus we get the implicit relation $A_{k,N}^2\geq B^2 k$. From 
this relation and an estimate of the parameter $A_{k,N}$ we
eventually deduce an upper bound for $k$. To conclude the argument
of Bourgain, the bounded part is uniformly estimated by
a classical concentration inequality and the rest is controlled by
the parameter $A_{k,N}$.

Notice that we only need tail inequalities to estimate $A_{k,N}$, that
is to control uniformly the norms of sub-matrices of $A$. This is still
a difficult task however because of a high complexity of the problem
and the lack of matching probability estimates; and a more
sophisticated argument has been developed 
in \cite{jams} to handle it.

\vspace{2ex}

We now pass to
the complementary case $n < N$, which is one of central points of the
present paper, and was announced in \cite{cras_allpt}. 

Let $A$ be an $n\times N$ random matrix with rows $X_1,
\ldots, X_n$ which are independent random 
centered and with covariance matrices equal to the
identity, but not necessarily identically distributed.
Clearly, $A$ is then not invertible.  The uniform concentration on the
sphere $U_N=S^{N-1}$ (which appeared in the definition of
$\delta_m(A/\sqrt n)$ for $m=N$) does not hold and the expressions in
(\ref{eq:delta-m-1}) are not uniformly small on $U_N = S^{N-1}$.  The
best one can hope for is that $A$ may be ``almost norm-preserving'' on
some subsets of $S^{N-1}$. This is true for subsets $U_m$, for some
$1\le m \le N$ and is indeed measured by $\delta_m (A/\sqrt n)$.

The parameter $\delta_m$ plays
a major role in the Compressive Sensing theory
and an important question is to 
bound it  from  above with high
probability,  for  some (fixed) $m$.
For example, it  can be directly used
to express  the so-called  Restricted Isometry Property
(RIP) (introduced by  E.~Candes and T.~Tao in \cite{CT1})
which in turn ensures that
every $m$-sparse vector $x$ can be reconstructed from its compression
$Ax$ with $n\ll N$ by the so-called $\ell_1$-minimization method.

For matrices with independent rows $X_1, \ldots, X_n$, questions on
the RIP were understood and solved in the case of Gaussian and
sub-gaussian measurements (see \cite{CT1}, \cite{MPT} and
\cite{BDDW}). When $X_1, \ldots, X_n$ are independent log-concave
isotropic random vectors, these questions remained open and this is
one of our motivation for this article.

For an $n \times N$ matrix $A$ and $m \le N$, the definition of
$\delta_m(A/\sqrt n) $ implies a uniform control of the norms of all
sub-matrices of $A/\sqrt n$ with $n$ rows and $m$ columns. 
Passing to transposed matrices,
it implies a uniform control of ${|P_I X_i|}$ over all $I\subset
\{1,2,\cdots, N\}$ of cardinality $m$ and $1\le i\le n$.  In order
to verify  a necessary condition
that for some $m$, $\delta_m(A/\sqrt n)$ is small with
high probability, one  needs  to get an upper estimate for $\sup
\{|P_I X|: |I| =m\}$ valid  with high probability.

The probabilistic inequality from \cite{Pa}
\begin{equation}
 \label{Paouris estimate}
 \p\left(|P_I X|\ge C \, t  \sqrt m \right) \le e^{-t \sqrt m}
\end{equation}
valid for $t\geq 1$ is optimal for each individual $I$, but it does
not allow to get directly (by a union bound argument) a uniform
estimate because the probability estimate does not match the
cardinality of the family of the $I$'s.
 Thus the first natural goal we
address in this paper is to get uniform tail estimates for some norms
of log-concave random vectors.

\vspace{2ex}

This heuristic analysis points out  to the main objective and
novelty of the present paper; namely  the study of high-dimensional
log-concave measures and a deeper understanding of such measures and
their convolutions via new tail estimates for norms of sums of
projections of log-concave random vectors.

\vspace{2ex}

To emphasize a uniform character of our tail estimates, for an integer
$N\ge 1$, an $N$-di\-men\-sio\-nal random vector $Z$, an integer $1 \le m
\le N$, and $t\ge 1$, we consider the event
\begin{equation}
  \label{eq:omega}
\Omega (Z, t, m, N) = 
\Big\{\sup_{{I \subset \{1,\ldots,N\}}\atop {|I| = m}}|P_IZ|\geq
Ct\sqrt{m}\log\Big(\frac{eN}{m}\Big)\Big\},
\end{equation}
where $C$ is a sufficiently large absolute constant.
Note that the cut-off level in this definition is of the order of the
median of the supremum for the exponential random vector.

\vspace{2ex}

Recall that  $X, X_1, \ldots, X_n$  denote    $N$-dimensional
independent log-concave  isotropic random vectors,
and $A$ is the $n\times N$ matrix whose rows are 
$X_1, \ldots, X_n$.
A chain of
main results of this paper provides estimates 
for $\Pr (\Omega (Z, t, m,N))$  in the cases  when 
\begin{description}
\item[(i)]  $Z= X $; and, more generally,
\item[(ii)]   $ Z = Y $ is a weighted sum $Y  =\sum_1^n x_iX_i$,
where $x=(x_i)_1^n\in\R^n$, with control of  
the Euclidean and supremum norms of $x$,
\item[(iii)] a uniform version of (ii) in the form of tail estimates
  for operator norms of sub-matrices of  $A$.
\end{description}

Our first main theorem solves the question  of uniform tail estimates
for projections of a log-concave random vector discussed above.

\begin{thm}
  \label{imprunifPaouris-intro} 
Let $X$ be  an $N$-dimensional
 log-concave isotropic random vector. 
For any $1 \le m \le N$ and $t\ge 1$,
$$
\Pr \Big(\Omega(X, t, m, N) \Big)
\leq \exp\Big(-t \sqrt{m}\log\Big(\frac{eN}{m}\Big)
/\sqrt{\log(em)}\Big).
$$
\end{thm}

The proof of 
the theorem is based on tail estimates for order statistics
of isotropic log-concave vectors.
By $(X^*(i))_i$ we denote the non-increasing rearrangement  of
$(|X(i)|)_i$. Combining 
(\ref{Paouris estimate})  with methods of  \cite{La}  and the formula
$\sup_{{I \subset \{1,\ldots,N\}}\atop {|I| = m}}|P_IX|=
\left(\sum_{i=1}^m X^*(i)^2 \right)^{1/2}$ will complete the argument.

\vspace{2ex}

Let us also mention that further applications (in
Section~\ref{section_convolutions}) of inequality of this type require
a stronger probability bound that involves a natural parameter
$\sigma_X(p)$ -- defined in (\ref{def_sigma}) -- determined by a
``$\ell_p$-weak'' behavior of the random vector $X$.

More generally, the next step provides  tail estimates for
Euclidean norms of weighted sums of independent isotropic 
log-concave random vectors.
Let $x=(x_i)_1^n\in\R^n$ and set $Y=\sum_1^n x_iX_i=A^* x$.
The key estimate used later, Theorem \ref{singlex}, provides
uniform estimates for the Euclidean norm of projections of $Y$.
Namely,  for every  $x\in\R^n$, 
$\Pr \Big(\Omega(Y, t, m, N) \Big)$ is exponentially small
with specific estimates depending on whether the ratio
$\|x\|_\infty/|x|$  is larger or smaller than $1/\sqrt m$. 
Since precise formulations of  probability estimates  are rather
convoluted we do not state them here and  we refer the reader
to Section \ref{section_convolutions}.

\vspace{2ex}

The last step of this chain  of results estimating probabilities
of  (\ref{eq:omega})  is connected with the family 
of parameters $A_{k, m}$, with $1\le k \le n$ and $1\le m\le N$,
defined by
\begin{equation}
  \label{eq:Akm}
A_{k,m} = \sup_{{J \subset \{1,\ldots,n\}}\atop {|J| = k}}
\sup_{x\in {U_m}} \Bigl(\sum_{j\in J} |\langle X_j,
x\rangle|^2\Bigr)^{1/2}
= \sup_{{J \subset \{1,\ldots,n\}}\atop {|J| = k}}
\sup_{x\in {U_m}} |P_J Ax|.
\end{equation}
That is, $A_{k,m}$ is the maximal operator norm over all sub-matrices
of $A$ with $k$ rows and $m$ columns (and for $m=N$ it obviously
coincides with (\ref{eq:AkN})).

Finding bounds on deviation of $A_{k,m}$ is one of our main goals.  To
develop an intuition of this result 
we state it below in a slightly less technical
form. Full details are contained in  Theorem \ref{est_akm}.

\medskip 
\begin{thm}
  \label{akm-intro}
For any $t \ge 1$ and $n\leq N$ we have
$$
\p\big(A_{k,m} \ge Ct \lambda \big) \le \exp(-t \lambda/\sqrt{\log (3m)}),
$$
where $\lambda = \sqrt{\log\log (3m)} \sqrt{m}\log(e N /m)
   + \sqrt{k}\log(en/k) $ and $C$ is a universal constant.
\end{thm}

The threshold value $\lambda$ is optimal,
up to the factor of $\sqrt{\log\log (3m)}$. Assuming additionally
unconditionality of the distributions of rows (or  columns), this factor can
be removed to get a sharp estimate (see \cite{ALLPT-unc}).

\vspace{2ex}

We make several comments about the proof.  
Set $  \Gamma = A^* $. Then
\[
    A  _{k,m} =
 \sup_{{I \subset \{1,\ldots,N\}}\atop {|I| = m}}
\sup_{x\in {U_k(\R^n)}} |P_I \Gamma x|
= \sup_{{I \subset \{1,\ldots,N\}}\atop {|I| = m}}
\sup_{x\in {U_k(\R^n)}} \Big|\Big(\sum x_i P_I  X_i\Big)\Big|.
\]
To bound $A_{k,m}$ one has then to  prove uniformity with respect to two
families of  different  character: 
one coming from the cardinality of the family
$\{I\subset\{1,\ldots,N\}:\,\ |I|=m\}$; and the other,
from the complexity of $U_k(\R^n)$.  This leads us to distinguishing
two cases, depending on the relation between $k$ and the quantity
\[
  k' = \inf\{\ell \ge 1\colon m\log (eN/m )\le
  \ell\log (en/\ell )\}. 
\]
First, if $k \ge k'$, we adjust the chaining argument similar to the
one from \cite{jams} to reduce the problem to the case $k \le k'$. In
this step we use the uniform tail estimate from Theorem
\ref{imprunifPaouris} for the Euclidean norm of the family of vectors
$\{P_IX\,:\, |I|=m\}$. Next, we use  a different  chain
decomposition of $x$ and apply 
Theorem~\ref{singlex}. 

\vspace{4ex}

As already alluded to, an independent interest of this paper lies
in upper bounds for  $\delta_m (A/\sqrt n)$ where $A$ is our $n\times
N$ random matrix. We presently return to this  subject to explain the
connections. 

\vspace{2ex}

The family $A_{k,m}$ plays a very essential role in studies of the 
Restricted Isometry constant,
which in fact applies even in a  more general setting.
Namely, for an arbitrary subset  $T \subset \R^N$ and $1 \le k \le n$
define the  parameter $A_k(T)$ by
\begin{equation}
  \label{akT} 
  A_k(T) = 
\sup_{{I \subset \{1,\ldots,n\}}\atop {|I| = k}} 
\sup_{y\in T} \Bigl(\sum _{i\in I}  
|\la X_i , y \ra |^2\Bigr)^{1/2}.
\end{equation}
Thus $A_{k,m}=A_k(U_m)$.  
The parameter $A_k(T)$  was studied  in \cite{M1} 
by means of Talagrand's  $\gamma$-functionals.

The following lemma reduces a concentration
inequality to a deviation inequality 
and hence is useful in studies of the RIP.
It is  based on an argument of truncation
similar to Bourgain's  approach presented earlier.

\begin{lemma}
  \label{dva-intro}
Let $X_1, \ldots, X_n$ be independent
isotropic random vectors in $\R^N$. Let $T \subset S^{N-1}$ be a
finite set.  Let $0<\theta < 1$ and $B\geq 1$.  Then with probability
at least $1-|T| \exp\left(- {3 \theta ^2 n}/{8 B^2} \right)$ one has
$$
 \sup_{y\in T} \left|\frac{1}{n} \sum_{i=1}^n(|\langle
 X_i, y\rangle|^2 - \E |\langle X_i, y\rangle|^2) \right|
 \leq  \theta + \frac{1}{n}\left(A_{k}( T) ^2
+ \E  A_{k}(T)^2\right),
$$
where $k\leq n$ is the largest  integer satisfying $k\leq
(A_{k}(T)/B)^2$. 
\end{lemma}

\vspace{2ex}

In this paper we focus on the compressive sensing setting where $T$
is the set of sparse vectors. The  lemma above  shows that after a
suitable discretisation, estimating $\delta_m$ or checking the RIP,
can be reduced to estimating $A_{k,m}$. This generalizes naturally
Bourgain's approach explained above for $m=N$.

Using the  lemma, we can show  
that if $0<\theta < 1$, $B\geq 1$, and $m \le N$
satisfy $  m \log (C N/ m) \leq 3 \theta ^2 n/16 B^2,$ then with
probability at least $1-  \exp\left( -  {3 \theta ^2 n}/{16 B^2}
\right)$ one has
\begin{equation*}
  \delta_m(A/\sqrt n) \leq  \theta + \frac{1}{n}\left(A_{k,m}^2
+ \E A_{k,m}^2\right),
\end{equation*}
where $k\leq n$ is the largest integer satisfying $k\leq
(A_{k,m}/B)^2$ (note that $k$ is a random variable).

Combining this with tail inequalities from Theorem \ref{akm-intro}
allows us to prove the following result on the RIP of matrices with
independent isotropic log-concave rows.

\begin{thm}
  \label{RIP-intro-thm}
Let $0<\theta < 1$, $1 \le n \le N$.
Let $A$ be an $n\times N$ random matrix with independent isotropic
log-concave rows. There exists $c(\theta)>0$ such that $\delta _m
(A/\sqrt n) \leq \theta$ with an overwhelming probability,  whenever
$$
m\log^2(2N/m) \log\log 3m \leq c(\theta) n.
$$
\end{thm}

The result is optimal, up to the factor $\log\log 3m$, as shown in
\cite{ALPT1}. As for Theorem \ref{est_akm}, assuming
unconditionality of the distributions of the rows, this factor can
be removed (see \cite{ALLPT-unc}).

\vspace{4ex}

The paper is organized as follows. In the next section
we collect the notation and necessary preliminary tools concerning
log-concave random variables. 
In Section~\ref{Section_New_Est}, given an isotropic log-concave
random vector $X$, we present several uniform tail estimates for
Euclidean norms of the whole family of projections of $X$ on
coordinate subspaces of dimension $m$. As already mentioned, these
estimates are based on tail estimates for order statistics of $X$.
The main result, Theorem \ref{imprunifPaouris}, provides a strong
probability bound in terms of the ``$\ell_p$-weak'' parameter
$\sigma_X(p)$ defined in (\ref{def_sigma}).  The proofs of the main
technical results, Theorems~\ref{estN} and \ref{imprunifPaouris}, are
given in Section~\ref{proof-section}.
Section~\ref{section_convolutions} provides tail estimates for
Euclidean norms of projections of weighted sums of independent
isotropic log-concave random vectors. The proof of the main
Theorem~\ref{singlex} is a combination of
Theorem~\ref{imprunifPaouris} and one-dimensional Proposition
\ref{singlex}.
In Section~\ref{submatrix} we prove 
the result announced above on deviation of $A_{k,m}$. Section~\ref{RIP} 
treats the Restricted Isometry Property and estimates of $\delta_m(A/\sqrt n)$. 
The last Section~\ref{proof-section} is devoted to the proofs of
technical results of Section~\ref{Section_New_Est}.

\medskip

\noindent
{\bf Acknowledgment:\ } The research on this project was partially
done when the authors participated in the Thematic Program on
Asymptotic Geometric Analysis at the Fields Institute in Toronto in
Fall 2010 and in the Discrete Analysis Programme at the Isaac Newton
Institute in Cambridge in Spring 2011. The authors wish to thank these
institutions  for their  hospitality and excellent working conditions.

\medskip

\section{Notation and preliminaries}
\label{notprel}

Let $L$ be an origin symmetric convex compact body in $\R^d$. This is
the unit ball of a norm that we denote by $\|\cdot\| _L $. Let
$K\subset \R^d$. We say that a set $\Lambda\subset K$ is an $\eps$-net
of $K$ with respect to the metric corresponding to $L$ if
$$
  K\subset \bigcup _{z\in \Lambda} (z+ \eps L).    
$$ 
In other words, for every $x\in K$ there exists $z\in \Lambda$ such
that $\|x-z\| _L \leq \eps$. We will mostly use $\eps$-nets in the
case $K=L$.  It is well-known (and follows by the standard volume
argument) that for every symmetric convex compact body $K$ in $\R^d$
and every $\eps >0$ there exists an $\eps$-net $\Lambda$ of $K$ with
respect to metric corresponding to $K$, of cardinality not exceeding
$(1+2/\varepsilon)^d$.  It is also easy to see that $\Lambda \subset
K\subset (1-\eps)^{-1}\ \mbox{conv}\Lambda.$ In particular, for any
convex positively 1-homogenous function $f$ one has
$$
    \sup _{x\in K} f(x) \leq (1-\eps)^{-1}\\  \sup _{x\in \Lambda} f(x) .  
$$

A random vector $X$ in $\R^n$ is called isotropic if  
$$
\E\langle X,y\rangle=0,\quad \E\,|\langle X, y \rangle|^{2}=|y|^{2}
\quad \mbox{\rm for all }
 y\in \R^{n},  
$$
in other words, if $X$ is centered and its covariance matrix 
$\E\, X\otimes X$ is the identity.

A random vector $X$ in $\R^n$ is called log-concave if for all compact
nonempty sets $A,B\subset \R^n$ and $\theta\in [0,1]$, $\Pr(X\in
\theta A+(1-\theta)B)\geq \Pr(X\in A)^{\theta}\Pr(X\in B)^{1-\theta}$.
By the result of Borell \cite{Bo} a random vector $X$ with full
dimensional support is log-concave if and only if it admits a
log-concave density $f$, i.e. such density for which
\begin{displaymath}
 f (\theta x +(1-\theta) y) \ge f (x)^\theta  f(y)^{1-\theta} 
\quad \mbox{ for all } x, y \in \R^n,\ \theta\in[0,1].
\end{displaymath}

It is known that any affine image, in particular any projection,  
of a log-concave random vector is log-concave. Moreover, if $X$ and $Y$ 
are independent log-concave random vectors then so is $X+Y$ 
(see \cite{Bo, DKH, Pr}). 

One important and simple model of a centered log-concave random variable
with variance 1 is the symmetric exponential random variable $E$ which
has density $f(t)=2^{-1/2}\exp (-\sqrt{2} |t|)$. In particular for
every $s>0$ we have $\p(|E|\geq s)=\exp(-s/\sqrt 2)$.

Every centered log-concave random variable $Z$, with variance 1 satisfies
a sub-exponential inequality: 
\begin{equation}\label{decay}
    \text{for every}\ s>0,\quad \p(|Z|\geq s)\leq C \exp(-s/C),
\end{equation}
where $C>0$ is an absolute constant (see \cite{Bo}).

\begin{defi}
\label{def:psi1_norm}
For a random variable $Z$ we define the $\psi_1$-norm by
$$
\|Z\|_{\psi_1}=\inf\left\{C>0\, : \, \, \E\exp\left({|Z|/C}\right)
\leq 2\right\}
$$
and we say that $Z$ is $\psi_1$ with constant $\psi$, if
$\|Z\|_{\psi_1}\leq \psi$.
\end{defi}

A consequence of (\ref{decay}) is that there exists an absolute
constant $C>0$ such that any centered log-concave random variable with
variance 1 is $\psi_1$ with constant $C$.

It is well known that the $\psi_1$-norm of a random variable may be
estimated from the growth of the moments. More precisely if a random
variable $Z$ is such that for any $p\geq 1$, $\|Z\|_p\le p K$,
for some $K >0$, then $\|Z\|_{\psi_1}\le cK$ where $c$ is an absolute 
constant.

By $|\cdot|$ we denote the standard Euclidean norm on $\R^n$ as well 
as the cardinality of a set. By $\la \cdot, \cdot \ra$ we denote 
the standard inner product on $\R^n$. 
We denote by $B_2^n$ and $S^{n-1}$ 
the standard Euclidean unit ball and unit sphere in $\R^n$.

A vector $x \in \R^n$ is called sparse or $k$-sparse for some 
$1\leq k\leq n$ if the cardinality of its support satisfies 
$|\supp x| \le k$. 

We let 
\begin{equation}
  \label{Uk-defin}
U_k =U_k(\R^n):= \{x \in S^{n-1} \colon x \mbox{ is } k
\mbox{-sparse}\}.
\end{equation}

For any subset  $I \subset \{1,\ldots,N\}$ let $P_I$ denote the orthogonal
projection on the coordinate subspace $\R^{I} := \{y\in \R^N\colon
\supp y \subset  I\}$.

We will use the letters $C, C_0, C_1, \ldots$, $c, c_0, c_1, \ldots$ 
to denote positive absolute constants whose values may differ at each 
occurrence. 

\section{New bounds for log-concave  vectors}
\label{Section_New_Est}

In this section we state several new estimates for Euclidean 
norms of log-concave random vectors. Proofs of Theorems 
\ref{estN} and \ref{imprunifPaouris} are given in  
Section~\ref{proof-section}.

We start with the following theorem, which was essentially proved by
Paouris in \cite{Pa}.  Indeed, it is a consequence of 
Theorem~8.2 combined with Lemma~3.9 in that paper,
after checking that Lemma~3.9 holds not only for convex bodies but for
log-concave measures as well.

\begin{thm}
\label{imprPaouris} 
For any $N$-dimensional log-concave random
vector $X$ and any $p \ge 1$ we have
\begin{equation}
\label{equ:imprPaouris} 
 (\Ex|X|^p)^{1/p}\leq C\Big((\Ex |X|^2)^{1/2}+\sup_{t\in 
 S^{N-1}}(\Ex|\langle t,X\rangle|^p)^{1/p}\Big), 
\end{equation}
where $C$ is an absolute constant. 
\end{thm}

\noindent {\bf Remarks. 1.\ } It is well known (cf. \cite{Bo}) that if
$Z$ is a log-concave random variable then
$$
(\Ex|Z|^p)^{1/p}\leq C\frac{p}{q}(\Ex|Z|^q)^{1/q}\quad \mbox{ for
}p\geq q\geq 2.  
$$
If $Z$ is symmetric one may in fact take $C=1$ (cf. Proposition 3.8 in
\cite{LaWo}) and if $Z$ is centered then denoting by $Z'$ an
independent copy of $Z$ we get for $p\geq q\geq 2$,
$$
(\Ex|Z|^p)^{1/p}\leq (\Ex|Z-Z'|^p)^{1/p}\leq \frac{p}{q}(\Ex|Z-Z'|^q)^{1/q}
\leq 2\frac{p}{q}(\Ex|Z|^q)^{1/q}. 
$$
Therefore if $X \in \R^N$ is isotropic log-concave then
$$
\sup_{t\in S^{N-1}}(\Ex|\langle 
 t,X\rangle|^p)^{1/p} 
\le p \sup_{t\in S^{N-1}}(\Ex|\langle 
 t,X\rangle|^2)^{1/2} = p.
$$
Also note that $(\Ex |X|^2)^{1/2}= \sqrt N$.
Combining  these estimates together with inequality (\ref{equ:imprPaouris}),
we get that
$(\Ex|X|^p)^{1/p}\leq C(\sqrt N+p)$. Using 
Chebyshev's inequality we conclude that there exists $C>0$  such that
for every isotropic log-concave random 
vector  $X \in \R^N$ and  every $s \ge 1$ 
\begin{equation}
  \label{paour_dev}
 \p\left(|X|\ge C \, s  \sqrt N \right) \le e^{-s \sqrt N}
  \end{equation}
which is Theorem 1.1 from \cite{Pa}.

\noindent
{\bf 2. \ }
It is well known and it follows from  \cite{Bo}  that for any $p\geq 1$,
$(\Ex |X|^{2p})^{1/2p}\leq C (\Ex|X|^p)^{1/p}$ where $C$ is an absolute
constant. From the comparison between the first and  second moment
it is clear that inequality (\ref{equ:imprPaouris}) is  an equivalence.
 Moreover,  there exists $C>0$ such that
 $$
\p\left(|X|\ge C\Bigl((\Ex |X|^2)^{1/2} + \sup_{t\in S^{N-1}}(\Ex|\langle 
 t,X\rangle|^p)^{1/p} \Bigr) \right) \le e^{-p}
$$
and
$$
 \p\left(|X|\ge \frac{1}{C}\Bigl((\Ex |X|^2)^{1/2} + \sup_{t\in S^{N-1}}
 (\Ex|\langle t,X\rangle|^p)^{1/p}\Bigr) \right )\ge
\min\Big\{\frac{1}{C}, e^{-p}\Big\}.
$$
The upper bound follows trivially from  Chebyshev's inequality.  The
lower bound is a consequence of Paley-Zygmund's inequality and comparison
between the $p$-th and $(2p)$-th moments of $|X|$.

\noindent
{\bf 3. \ }
Since for any Euclidean norm $\|\cdot \|$ on $\er^N$ there exists a linear
map $T$ such that $\|x\|=|Tx|$ and the class of log-concave random vectors is
closed under linear transformations, Theorem~\ref{imprPaouris} implies
that for any $N$-dimensional log-concave vector $X$, any Euclidean norm
$\|\cdot \|$ on $\er^N$ and
$p \ge 1$ we have
\begin{equation*}
\label{weak_strong}
(\Ex \|X\|^p)^{1/p}\leq
C\Big((\Ex \|X\|^2)^{1/2}+\sup_{\|t\|_{*}\leq 1}(\Ex|\langle
t,X\rangle|^p)^{1/p}\Big), 
\end{equation*}
where $(\er^N,\|\cdot \|_{*})$ is the  dual space to $(\er^N,\|\cdot \|)$.
It is an open problem whether  such an inequality holds for arbitrary
norms -- see \cite{La1} for a discussion of this question and for related 
results.

\medskip

We now introduce our main technical notations. 
For a random vector $X=(X(1),\ldots,X(N))$ in $\er^N$, $p \ge 1$
and $t >0$  consider the functions 
\begin{equation}
  \label{def_sigma}
\sigma_X(p)=\sup_{t\in S^{N-1}}(\Ex|\langle t,X\rangle|^p)^{1/p}
\end{equation}
and
\[
N_X(t)=\sum_{i=1}^N\ind{X(i)\geq t}.
\]
That is, $N_X(t)$ is equal to the number of coordinates of $X$ larger
than or equal to $t$.
By $\sigma_{X}^{-1}$ we denote the inverse of $\sigma_X$ i.e.,
\[
\sigma_X^{-1}(s)=\sup\{t\colon \sigma_X(t)\leq s\}.
\]
Remark 1 after Theorem \ref{imprPaouris} implies that for isotropic
vectors $X$,  
$\sigma_X(tp)\leq 2t\sigma_X(p)$
for $p\geq 2,t\geq 1$ and $\sigma_{X}^{-1}(2ts)\geq t\sigma_X^{-1}(s)$ for
$t,s\geq 1$.

We also denote a nonincreasing rearrangement of $|X(1)|,\ldots,|X(N)|$ 
by $X^*(1)\geq X^*(2)\geq\ldots\geq X^*(N)$.

\smallskip

One of the main technical tools of this paper says:

\begin{thm}
\label{estN}
For any $N$-dimensional log-concave isotropic random vector $X$,
$p\geq 2$  and $ t\geq C\log\Big({Nt^2}/{\sigma_X^2(p)}\Big)$
we have
\[
\Ex(t^2N_X(t))^p\leq (C\sigma_X(p))^{2p}, 
\]
where $C$ is an absolute positive constant. 
\end{thm}

We apply Theorem~\ref{estN} to obtain probability estimates on 
order statistics $X^*(i)$'s.

\begin{theorem}
\label{orderstat}
For any $N$-dimensional log-concave random isotropic vector $X$, 
any  $1\leq \ell\leq N$ and $t\geq C\log (eN/\ell)$,
\[
\Pr(X^*(\ell)\geq t)\leq
\exp\Big(-\sigma_X^{-1}\Big(\frac{1}{C}t\sqrt{\ell}\Big)\Big), 
\]
where $C$ is an absolute positive constant. 
\end{theorem}

\begin{proof}
Observe that 
$\sigma_{-X}(p)=\sigma_X(p)$  and that 
$X^*(\ell)\geq t$ implies that $N_X(t)\geq \ell/2$ or
$N_{-X}(t)\geq \ell/2$. So by Chebyshev's inequality and Theorem
\ref{estN},
\[
 \Pr(X^*(\ell)\geq t)\leq \Big(\frac{2}{\ell}\Big)^p(\Ex N_X(t)^p+\Ex
 N_{-X}(t)^p)\leq \Big(\frac{C'\sigma_X(p)}{t\sqrt{\ell}}\Big)^{2p}
\]
provided that $t\geq C''\log(Nt^2/\sigma_X^2(p))$, where $C', C''$ are  
absolute positive constants. To conclude the proof 
it is enough to take $p=\sigma_{X}^{-1}(\frac{1}{C' e}t\sqrt{\ell})$ and 
to notice that the restriction on $t$ follows by the condition 
$t\geq C\log(eN/\ell)$.
\end{proof}

We can now state one of the main results of this paper. 

\begin{thm}
\label{imprunifPaouris}
Let $X$ be an isotropic log-concave  random vector in $\er^N$  and  $m \le
N$.  For any $t\geq 1$,
\[
\Pr\Bigg(\sup_{{I \subset \{1,\ldots,N\}}\atop {|I| = m}}|P_IX|\geq
Ct\sqrt{m}\log\Big(\frac{eN}{m}\Big)\Bigg)
\leq
\exp\left(-\sigma_{X}^{-1}\left(\frac{t\sqrt{m}
\log\left(\frac{eN}{m}\r)}{\sqrt{\log(em/m_0)}}\r)\r),
\]
where $C$ is an absolute positive constant and  
\[
m_0=m_0(X,t)=\sup\Big\{k\leq m\colon\ k\log\Big(\frac{eN}{k}\Big)\leq
\sigma_{X}^{-1}\Big(t\sqrt{m}\log\Big(\frac{eN}{m}\Big)\Big)\Big\}.  
\]
\end{thm}

\noindent
{\bf Remark.} We believe that the probability estimate should not contain
any logarithmic term in the denominator, but it seems that our methods fail
to show it. However it is not crucial in the  sequel.

\medskip

Since $\sigma_X(p)\leq p$, Theorem \ref{imprunifPaouris-intro} is an
immediate consequence of Theorem~\ref{imprunifPaouris}.

\section{Tail estimates for projections of sums of log-con\-ca\-ve
  random vectors} 
\label{section_convolutions}

We shall now study consequences that the results of 
Section~\ref{Section_New_Est} have for tail estimates 
for Euclidean norms of projections of sums of log-concave 
random vectors. Namely, we investigate the behavior of a random 
vector $Y=Y_x=\sum_{i=1}^n x_i X_i$,  where $X_1,\ldots, X_n$ are 
independent isotropic log-concave random vectors in $\R^N$
and  $x = (x_i)_1^n \in \er^n$ is  a fixed vector. We provide  
uniform bounds on projections of such a vector. 
We start with the following proposition.

\begin{prop}
\label{G-K}
Let $X_1,\ldots,X_n$ be independent isotropic log-concave random vectors in 
$\R^N$, $x = (x_i)_1^n \in \er^n$, and $Y=\sum_{i=1}^n x_i X_i$. Then
for every $p\ge 1$ one has
\[
  \sigma _Y(p) = \sup_{t\in S^{N-1}}(\Ex|\langle t,Y\rangle|^p)^{1/p}
     \leq C(\sqrt{p}|x|+p\|x\|_{\infty}),  
\]
where $C$ is an absolute positive constant. 
\end{prop}

\smallskip

\begin{proof}
For every $t\in S^{N-1}$ we have
\[
 \langle t, Y\rangle =\sum_{i=1}^n x_i\langle t, X_i\rangle.
\]
Let $E_i$ be independent symmetric exponential random variables with
variance $1$.
Let $t\in S^{N-1}$ and $x = (x_i)_1^n \in \er^n$. The variables 
$Z_i=\langle t, X_i\rangle$ are one dimensional centered 
log-concave with variance $1$, therefore by (\ref{decay}) for every 
$s>0$ one has 
$$
 \p\left(|Z_i|\geq s \r)  
 \leq C_0\ \p\left(|E_i|\geq s/ C_0 \r) .
$$
Let $(\varepsilon_i)$ be independent Bernoulli $\pm1$ random variables,
independent also from $(Z_i)$. 
A classical symmetrization argument and Lemma~4.6 of \cite{LT} imply that 
there exists $C$ such that
$$
  (\Ex|\langle t,Y\rangle|^p)^{1/p} \leq 2\Big(\Ex\Big|\sum_{i=1}^n x_i 
  \varepsilon_i Z_i\Big|^p\Big)^{1/p} \leq C\Big(\Ex\Big|\sum_{i=1}^n 
  x_i E_i\Big|^p\Big)^{1/p} . 
$$
The well-known estimate (which follows e.g. from Theorem~1 in \cite{GK})
$$
  \left(\E \Big|\sum_{i=1}^n x_i E_i\Big|^p\r)^{1/p} \leq 
  C(\sqrt{p}|x|+p\|x\|_{\infty}) 
$$
concludes the proof.  
\end{proof}

\bigskip

\begin{cor}
\label{orderstatconv}
Let $X_1, \ldots, X_n$, $x$  and $Y$ be as in Proposition~\ref{G-K} and 
$1\le \ell\le N $. 
 Then for any $t\geq C|x|\log\left(\frac{eN}{\ell}\r)$ one has 
\[
\Pr(Y^*(\ell)\geq t)\leq
\exp\left(-\frac{1}{C}\min\left\{\frac{t^2\ell}{|x|^2},\frac{t\sqrt{\ell}}  
{\|x\|_{\infty}} \r\}\r), 
\]
where $C$  is an absolute positive constant. 
\end{cor}

\begin{proof}
The vector $Z=Y/|x|$ is isotropic and log-concave. Moreover by 
Proposition~\ref{G-K} we have
\[
\sigma_{Z}(p)=\frac{1}{|x|}\sigma_Y(p)\leq
C_1\Big(\sqrt{p}+p\frac{\|x\|_{\infty}}{|x|}\Big).
\]
Therefore for every $t\geq C_1$ 
\[
\sigma_Z^{-1}(t)\geq
\frac{1}{C_2}\min\Big\{t^2,\frac{|x|}{\|x\|_{\infty}}t\Big\}
\]
and by Theorem~\ref{orderstat} we get for every 
$t\geq C_3 |x|\log\left(\frac{eN}{\ell}\r)$ 
\begin{align*}
\Pr(Y^*(\ell)\geq t)&=\Pr\Big(Z^*(\ell)\geq \frac{t}{|x|}\Big)\leq
\exp\Big(-\sigma_{Z}^{-1}\Big(\frac{t\sqrt{\ell}}{C_4 |x|}\Big)\Big)\\
&\leq
\exp\Big(-\frac{1}{C}\min\Big\{\frac{t^2\ell}{|x|^2},\frac{t\sqrt{\ell}}
 {\|x\|_{\infty}}\Big\}\Big).
\end{align*}
\end{proof}

The next theorem provides uniform estimates for the Euclidean norm of 
projections of sums $Y_x$, considered above, in terms of the Euclidean and 
$\ell_{\infty}$ norms of the vector $x\in \R^n$.

\begin{thm}
\label{singlex}
Let $X_1,\ldots,X_n$ be independent isotropic log-concave random vectors 
in $\R^N$, $x = (x_i)_1^n \in \er^n$, and  $Y=\sum_{i=1}^n x_i X_i$. 
Assume that $|x|\leq 1$, $\|x\|_{\infty}\leq b\leq 1$ and let $1\le m \le N$. 
\\ \\
i) If $b\geq \frac{1}{\sqrt{m}}$ then for any $t\geq 1$ 
\[
\Pr\Bigg(\sup_{{I \subset \{1,\ldots,N\}}\atop {|I| = m}}|P_IY|\geq
Ct\sqrt{m}\log\Big(\frac{eN}{m}\Big)\Bigg)
\leq
\exp\Bigg(-\frac{t\sqrt{m}\log\Big(\frac{eN}{m}\Big)}
{b\sqrt{\log(e^2b^2m)}}\Bigg) ;
\]
ii) if $b\leq \frac{1}{\sqrt{m}}$ then for any $t\geq 1$
\begin{align*}
\Pr\Bigg(\sup_{{I \subset \{1,\ldots,N\}}\atop {|I| = m}}|P_IY|\geq
Ct&\sqrt{m}\log\Big(\frac{eN}{m}\Big)\Bigg)
\\
&\leq
\exp\Big(-\min\Big\{t^2m\log^2\Big(\frac{eN}{m}\Big),
   \frac{t}{b}\sqrt{m}\log\Big(\frac{eN}{m}\Big)\Big\}\Big) , 
\end{align*}
where $C$ is an absolute positive constant. 
\end{thm}

\medskip

\noindent
{\bf Remark.}
  Basically the same proof as the one given below shows that in i) the term 
 $\sqrt{\log(e^2b^2m)}$ may be replaced by $\sqrt{\log(e^2b^2m/t^2)}$ 
 and the condition $b\geq \frac{1}{\sqrt{m}}$
  by $b\geq \frac{t}{\sqrt{m}}$. We omit the details.

\bigskip

The proof of Theorem \ref{singlex} is based on Theorem
\ref{imprunifPaouris}. Let us first note that we may assume that 
vector $Y$ is isotropic, i.e. $|x|=1$. Indeed, we may find vector 
$y=(y_1,\ldots,y_{\ell})$ such that $\|y\|_{\infty}\leq b$ and 
$|x|^2+|y|^2=1$ and take $Y'=\sum_{i=1}^{\ell} y_iG_i$, 
where $G_i$ are i.i.d. canonical $N$-dimensional Gaussian vectors, 
independent of vectors $X_i$'s. Then the vector $Y+Y'$ is isotropic, 
satisfies assumptions of the theorem and for any $u>0$,
\[
  \Pr\bigg(\sup_{{I \subset \{1,\ldots,N\}}\atop {|I| = m}}|P_IY|\geq 
  u\bigg)\leq 2\Pr\bigg(\sup_{{I \subset \{1,\ldots,N\}}\atop {|I| = 
  m}}|P_I(Y+Y')|\geq u\bigg).
\]

Similarly as in the proof of Corollary \ref{orderstatconv}, 
for $t\geq C$ we have
\begin{equation}\label{inverse}
\sigma_Y^{-1}(t)\geq
\frac{1}{C}\min\left\{t^2,\frac{t}{\|x\|_{\infty}}\r\}.
\end{equation}
This allows us to estimate the quantity $m_0$ in Theorem \ref{imprunifPaouris}.
For $1/\sqrt{m}\leq b\leq 1$ define $m_1=m_1(b)>0$ by the equation
\begin{equation}
\label{def_m0}
m_1(b)\log\left(\frac{eN}{m_1(b)}\r)=\frac{\sqrt{m}}
    {b}\log\Big(\frac{eN}{m}\Big).
\end{equation}

One may show that $m_1(b)\sim
\frac{\sqrt{m}}{b}\log(\frac{eN}{m})/\log(\frac{eNb}{\sqrt{m}})$,
we will however need only the following simple estimate.

\begin{lemma}
\label{est_m0}
If  $1/\sqrt{m}\leq b\leq 1$ then $\log(m/m_1(b))\leq 2\log(eb\sqrt{m})$.
\end{lemma}

\begin{proof}
Let $f(z)=z\log(eN/z)$. Using  $1/\sqrt{m}\leq b$ we observe 
\begin{align*}
 f\Big(\frac{1}{e^2b^2}\Big)&=\frac{1}{e^2b^2}\Big(\log\Big(\frac{eN}{m}\Big)+
 \log(me^2b^2)\Big)\leq
\frac{\sqrt{m}}{e^2b}\log\Big(\frac{eN}{m}\Big)+\frac{1}{e^2b^2}\sqrt{m}eb
\\
&\leq \frac{\sqrt{m}}{b}\log\Big(\frac{eN}{m}\Big)\Big(\frac{1}{e^2}
          +\frac{1}{e}\Big)< 
\frac{\sqrt{m}}{b}\log\Big(\frac{eN}{m}\Big)=f(m_1(b)).
\end{align*}
Since $f$ increases on $(0,N]$, we obtain 
$m_1(b)\geq (eb)^{-2}$, which implies the result. 
\end{proof}

\begin{proof}[Proof of Theorem \ref{singlex}] 
As we noticed after remark following Theorem~\ref{singlex}, 
without loss of generality we may assume that $|x|=1$, 
i.e. that $Y$ is isotropic.

i)  Assume $b\geq 1/\sqrt{m}$.
 By (\ref{inverse}) for every $t\geq C/(\sqrt{m}\log (eN/m))$ we have 
\begin{equation}\label{six}
\sigma_Y^{-1}\Big(t\sqrt{m}\log\Big(\frac{eN}{m}\Big)\Big) \ge 
\frac{t}{Cb}\sqrt{m}\log\Big(\frac{eN}{m}\Big) . 
\end{equation} 
By (\ref{def_m0}) it follows that for every  $t\geq |x|=1$  
\begin{displaymath}
m_1(b)\log\Big(\frac{eN}{m_1(b)}\Big) \le
\sigma_Y^{-1}\Big(Ct\sqrt{m}\log\Big(\frac{eN}{m}\Big)\Big).
\end{displaymath}
By the definition of $m_0$, given in Theorem~\ref{imprunifPaouris}, 
 this implies that $m_0(Y,Ct) \ge
\lfloor m_1(b)\rfloor$, and since $m_0(Y,Ct) \ge 1$ we get
$m_0(Y,Ct) \ge m_1(b)/2$. By Lemma~\ref{est_m0} this yields
$\log(em/m_0(Y,Ct)) \le 2 + 2\log(eb\sqrt{m}) \le
4\log(eb\sqrt{m})$.

Writing $t=t' \sqrt{2 \log (e^2 b^2 m)}$ and applying (\ref{six}) 
we obtain 
$$
 \sigma_Y^{-1}\left(\frac{t\sqrt{m}\log\left(\frac{eN}{m}\r)}{\sqrt{
 \log\frac{e m}{m_0(Y,Ct)}}}\r) \ge \sigma_Y^{-1}\left(t'\sqrt{m}
 \log\left(\frac{eN}{m}\r) \r) \geq 
 \frac{t'}{Cb}\sqrt{m}\log\Big(\frac{eN}{m}\Big) .
$$
Theorem~\ref{imprunifPaouris} applied to $Ct$ instead of $t$ implies 
the result (one needs to adjust absolute constants).

\smallskip

\noindent
ii) Assume $b\leq 1/\sqrt{m}$. By (\ref{inverse}) for every $t\geq C|x|$ 
we have 
\begin{align*}
\sigma_Y^{-1}\Big(t\sqrt{m}\log\Big(\frac{eN}{m}\Big)\Big)
&\ge \frac{1}{C}\min\Big\{t^2m\log^2\Big(\frac{eN}{m}\Big),
\frac{t}{b}\sqrt{m}\log\Big(\frac{eN}{m}\Big)\Big\}\\
& \ge \frac{t}{C} m\log\Big(\frac{eN}{m}\Big),
\end{align*}
which by the definition of $m_0$ implies that $m_0(Y,Ct) = m$. 
As in part i), Theorem~\ref{imprunifPaouris} implies the result.
\end{proof}

\section{Uniform bounds for norms of sub-matrices}
\label{submatrix}

In this section we establish uniform estimates for norms of
submatrices of a random matrix, namely for the quantity $ A_{k,m}$
defined below.

Fix integers $n$ and $N$.  
Let $X_1, \ldots, X_n \in \R^N$
be independent log-concave isotropic random vectors.
Let $ A $ be the 
$n\times N$ random matrix with 
rows $X_1, \ldots, X_n$.

For any subsets $ J\subset \{1,\ldots,n\}$ and $ I \subset
\{1,\ldots,N\}$, by $ A (J, I)$ we denote the submatrix of $ A $
consisting of the rows indexed by elements from $J$ and the columns
indexed by elements from $I$.

Let $k\le n$ and $m\le N$. 
We  define the parameter $ A_{k,m}$ by
\begin{equation}
\label{akmcorr}
   A_{k,m} = 
  \sup  \| A  (J, I)\|_{\ell_2^m \to \ell_2^k},
\end{equation}
where the supremum is taken over all subsets
$  J\subset \{1,\ldots,n\}$ and $ I \subset \{1,\ldots,N\}$  with
cardinalities  $ |J| = k, |I|= m$.
That is, $  A_{k,m}$ is  the maximal operator norm 
of a submatrix of $ A $ with $k$ rows and $m$ columns.

It is often more convenient to work with matrices with log-concave
columns rather than rows, therefore in this section we fix the
notation 
$$
\Ga = A  ^*.
$$  
Thus $\Ga$ is an $N\times n$ matrix with
columns $X_1, \ldots, X_n$. In particular, given $x\in \R^n$ the
sum $Y=\sum_{i=1}^n x_i X_i$ considered in Section
\ref{section_convolutions} satisfies $Y = \Ga x$.  Clearly,
$$
 \Ga (I, J)= \Bigl( A  (J, I) \Bigr)^*
$$
so that, recalling that $U_k$ was defined in (\ref{Uk-defin}), we have 
\begin{equation}
  \label{eq:akm-def}
    A  _{k,m} = \Ga _{m,k} =
  \sup\{|P_{I}\Ga x|\colon I\subset\{1,\ldots,N\},\ |I|=m,\ x\in U_k\}.
\end{equation}

\bigskip

Define  $\lam_{k,m}$ and  $\lam_{m}$ by
\begin{equation}
  \label{lam_km}
    \lam _{k,m} = \sqrt{\log\log (3 m)} \
\sqrt{m}\log\Big(\frac{e\max\{N,n\}}{m}\Big)
+ \sqrt{k}\log\Big(\frac{en}{k}\Big), 
\end{equation}
and
\begin{equation}
  \label{lam_m}
    \lam _{m} =\frac{ \sqrt{\log\log (3 m)} \
  \sqrt{m}}{\sqrt{\log (3m)}} \ \log\Big(\frac{e\max\{N,n\}}{m}\Big). 
\end{equation}

The following theorem is our main result providing estimates for the
operator norms of submatrices of $ A $ (and of $\Ga$).  
Its first part in the case $n \le N$ was stated as
Theorem~\ref{akm-intro}.

\begin{theorem}
\label{est_akm} 
There exists a positive absolute constant $C$ such that for any
positive integers $n, N$, $k\le n$, $m\le N$ and any $t \ge 1$ one has
$$
  \p \left( A_{k,m} \ge C t \lam _{k,m} \right)  \le 
  \exp\left(-\frac{t  \lam _{k,m} }{\sqrt{\log (3m)}}\right). 
$$
In particular, there exists an absolute positive constant
$C_1$ such that for every $t\geq 1$ and for every $m\leq N$ one has 
\begin{equation}
\label{theomakm}
\p\left(\exists k \, \, \, A_{k,m} \ge C_1 t \lam _{k,m} \right) 
\le \exp\left(- t \lam _m \right).
\end{equation} 
\end{theorem}

\bigskip

First we show the ``in particular" part, which is easy. 

\medskip

\noindent
{\it Proof of inequality (\ref{theomakm}).}
The main part of the theorem implies that for every $t\geq 1$  
$$
 p_m:= \p\left(\exists k \, \, \, A_{k,m} \ge C t \lam _{k,m} \right) 
  \le \sum_{k=1}^n \exp\left(-t \lam _{k,m} /\sqrt{\log (3m)} \right). 
$$
Thus if $m > \log ^3 n$ then for every $t\geq 100$ one has  
$$
   p_m \leq n  \exp\left(- t \lam _{m} \right) 
   \leq  \exp\left(- (t/2)\ \lam _{m} \right) .
$$ 
If $m\leq \log ^3 n$ (in particular $n\geq 3$) then for every $t\geq
100$ one has   
$$
   p_m \leq 
     \sum_{k\leq (\log n)^2} \exp\left(-t \lam _{k,m}/ \sqrt{\log (3m)} \right)
   + \sum_{k\geq (\log n)^2} \exp\left(-t \lam _{k,m}/ \sqrt{\log (3m)} \right)
$$ 
$$
  \leq (\log n)^2  \exp\left(- t \lam _{m} \right) +  \exp\left(- t \lam _{m} 
  \right) \sum_{k\geq (\log n)^2} \exp\left(-t a _{k,m} \right) , 
$$
where 
$$
  a_{k,m} = \frac{\sqrt{k}}{\sqrt{\log (3m)}}\log\Big(\frac{en}{k}\Big).
$$
Since $m\leq \log ^3 n$, we obtain for every $t\geq 100$ 
$$
 p_m \leq \exp\left(- (t/2) \lam _{m} \right) + \exp\left(- t \lam _{m}
 \right) \leq \exp\left(- (t/4) \lam _{m} \right) . 
$$
The result follows by writing $t=100 t'$ and by adjusting absolute constants.
\qed

\bigskip

Now we prove the main part of the theorem. Its proof 
consists of two steps 
that depend  on the relation 
between $m$ and $k$. The Step I is applicable if
$m\log\Big(\frac{eN}{m}\Big) < k\log\Big(\frac{en}{k}\Big)$,
and it reduces this case  to the second complementary
case 
$m\log\Big(\frac{eN}{m}\Big) \geq k\log\Big(\frac{en}{k}\Big)$.
The latter case  will then  be 
treated in Step II. To make this  reduction we define  $k'$ as 
follows
\begin{equation}
  \label{eq:k'-def}
  k' = \inf\Big\{\tilde{k}\in \N\ : \ \tilde k \leq n\ \mbox{ and } \ 
  \tilde{k}\log\Big(\frac{en}{\tilde{k}}\Big)   \ge
  m\log\Big(\frac{eN}{m}\Big)   
 \Big\}  
\end{equation}
(of course if 
the set in (\ref{eq:k'-def}) is empty, 
we immediately pass to Step II). 

\medskip

\subsection
{Step I: $k\log\Big(\frac{en}{k}\Big) > m\log\Big(\frac{eN}{m}\Big)$, 
in particular $k \ge k'$.}

\begin{prop}
  \label{step-1}
Assume that $k \ge k'$. Then for any $t \ge 1$ we have
\begin{align}
\label{small_m_ineqen_gen}
\sup_{{I \subseteq \{1,\ldots,N\}}\atop {|I| = m}} \sup_{x \in U_k}
|P_I \Gamma x| \le 
C\Big(&\sup_{{I \subseteq \{1,\ldots,N\}}\atop {|I| = m}} \sup_{x \in
  U_{k'}} |P_I \Gamma x| \nonumber\\
&+t\sqrt{m}\log\Big(\frac{eN}{m}\Big)
       +t\sqrt{k}\log\Big(\frac{en}{k}\Big)\Big)
\end{align}
with probability at least 
\begin{equation}
  \label{small_m_probab}
1 - n\exp\Big(-t \frac{\sqrt m \log (eN/m) +
\sqrt{k}\log(en/k)}{\sqrt {\log em}}\Big) -
\exp\Big(-tk'\log\Big(\frac{en}{k'}\Big)\Big),
\end{equation}
where $C$ is a positive absolute constant. 
\end{prop}

\medskip

The proof of Proposition \ref{step-1} is based on the ideas from \cite{jams}. 
We start it with the following fact.

\begin{prop}
\label{estpsi1}
Let $(X_i)_{i\leq n}$ be independent centered random vectors in 
$\er^N$ and $\psi>0$ be such that
\[
\Ex\exp\Big(\frac{|\langle X_i,\theta\rangle|}{\psi}\Big)\leq 2 \quad 
\mbox{ for all }i\leq n, \theta\in S^{N-1}.
\]
Then for $1\leq p\leq n$ and $t\geq 1$ with probability at least 
$1-\exp(-tp\log(en/p))$ the following holds: \\ 
for all $y,z\in U_p$ and all $E,F\subset \{1, \ldots, n\}$ with $E\cap
F=\emptyset$,  
\[
 \Big|\Big\langle\sum_{i\in E} y_i X_i,\sum_{j\in F}z_j X_j \Big\rangle\Big|
 \leq 20 tp\log\Big(\frac{en}{p}\Big)\psi\max_{i\in E}|y_i|\Big(\sum_{j\in F}
 z_j^2\Big)^{1/2} \Ga_p,
\]
where $\Ga_p:=\max_{x\in U_p}|\Ga x| = \max_{x\in U_p}|\sum_{i=1}^n x_i X_i|$. 
\end{prop}

\begin{proof} In this proof we  use for simplicity  the notation
  $[n]=\{1,\ldots, n\}$.   First let us fix sets $E,F\subset[n]$ with
  $E\cap F=\emptyset$. Since we consider $y, z\in U_p$, without loss
  of generality we may assume that $|E|,|F|\leq p$.  For $z\in U_p$
  denote
\[
  Y_F(z)=\sum_{j\in F}z_j X_j \quad \mbox{and} \quad
  Z_F(z)=\frac{Y_F(z)}{|Y_F(z)|} 
\]
(if $Y_F(z)=0$ we set $Z_F(z)=0$). 
For any $y,z\in U_p$  we have
\begin{align*}
\Big|\Big\langle\sum_{i\in E} y_i X_i,\sum_{j\in F}z_j X_j \Big\rangle\Big|
&\leq \sum_{i\in E}|y_i|\Big|\Big\langle X_i, Y_F(z)\Big\rangle\Big|
\\
&\leq\max_{i\in E}|y_i|\, |Y_F(z)|\sum_{i\in E}\Big|\Big\langle X_i,
Z_F(z)\Big\rangle\Big|.  
\end{align*}
The random vector $Z_F(z)$
is independent from the vectors $X_i$'s, $i\in E$, moreover
$|Z_F(z)|\leq 1$ and  
$|Y_F(z)|\leq (\sum_{j\in F}z_j^2)^{1/2} \Ga_p$. 
Therefore for any $z\in U_p$ and $u>0$,
\begin{align*}
\Pr \Bigl( \exists_{y\in U_p}\
\Big|\Big\langle\sum_{i\in E} &y_i X_i,\sum_{j\in F}z_j X_j \Big\rangle\Big|>
u \psi \max_{i\in E}|y_i|\Big(\sum_{j\in F}z_j^2\Big)^{1/2} \Ga_p   \Bigl) 
\\
&\leq
\Pr\Big(\sum_{i\in E}\Big|\Big\langle X_i, Z_F(z)\Big\rangle\Big|\geq u \psi\Big)
\\
&\leq e^{-u}\Ex \exp\Big(\sum_{i\in E}\frac{|\langle X_i, Z_F(z)\rangle|}{\psi}\Big)\leq
2^{|E|}e^{-u}\leq 2^pe^{-u}.
\end{align*}
Let $N_F$ denote a $1/2$-net in the Euclidean metric in $B_2^n\cap \er^F$ of cardinality  
at most $5^{|F|}\leq 5^p$. We have
\begin{align*}
p_{E,F}&(u)
\\
&:=\Pr\Big(\exists_{y\in U_p} \exists_{z\in U_p}\
\Big|\Big\langle\sum_{i\in E} y_i X_i,\sum_{j\in F}z_j X_j \Big\rangle\Big|>2
u \psi \max_{i\in E}|y_i|\Big(\sum_{j\in F}z_j^2\Big)^{1/2} \Ga_p\Big)
\\
&\leq \Pr\Big(\exists_{y\in U_p} \exists_{z\in N_F}\
\Big|\Big\langle\sum_{i\in E} y_i X_i,\sum_{j\in F}z_j X_j \Big\rangle\Big|>
u \psi \max_{i\in E}|y_i|\Big(\sum_{j\in F}z_j^2\Big)^{1/2} \Ga_p\Big)
\\
&\leq \sum_{z\in N_F}\Pr\Big(\exists_{y\in U_p}\
\Big|\Big\langle\sum_{i\in E} y_i X_i,\sum_{j\in F}z_j X_j \Big\rangle\Big|>
u \psi \max_{i\in E}|y_i|\Big(\sum_{j\in F}z_j^2\Big)^{1/2} \Ga_p\Big)
\\
&\leq 10^pe^{-u}.
\end{align*}
Hence
\begin{align*}
  \Pr\Big(&\exists_{y\in U_p} \exists_{z\in U_p}\exists_{E,F\subset[n], 
  |E|,|F|\leq p,E\cap F=\emptyset}
\\
  &\Big|\Big\langle\sum_{i\in E} y_i X_i,\sum_{j\in F}z_j X_j \Big\rangle\Big|>2
  u \psi \max_{i\in E}|y_i|\Big(\sum_{j\in F}z_j^2\Big)^{1/2} \Ga_p\Big)
\\
  &\phantom{aaaa}\leq \sum_{E,F\subset[n], |E|,|F|\leq p, E\cap
    F=\emptyset}p_{E,F}(u)\leq  
  \binom{n}{p}\binom{n}{p} 10^p e^{-u} \leq e^{-u/2}, 
\end{align*}
provided that $u\geq 10p \log (ne/p)$. This implies the desired result. 
\end{proof}

Before formulating the next proposition we recall the following 
elementary lemma (see e.g. Lemma 3.2 in \cite{jams}).

\begin{lemma}
\label{dec}
Let $x_1, \ldots, x_n \in \R^N$, then
$$
 \sum_{i\ne j} \langle x_i, x_j \rangle
\le 4 \max_{E \subset \{1,\ldots,n\}}
\sum_{i\in E}
  \sum_{j\in E^c} \langle x_i, x_j \rangle.
  $$
\end{lemma}

\begin{prop}
\label{estpsi2}
Let $(X_i)_{i\leq n}$ be independent centered random vectors in 
$\er^N$ and $\psi>0$ be such that
\[
\Ex\exp\Big(\frac{|\langle X_i,\theta\rangle|}{\psi}\Big)\leq 2 \quad 
\mbox{ for all }i\leq n, \theta\in S^{N-1}.
\]
Let $p\leq n$ and $t\geq 1$. Then with probability at least 
$1-\exp(-tp\ln(en/p))$ for all $x\in U_p$, 
\[
  \Big|\sum_{i=1}^n x_i X_i\Big|\leq C \Big(|x|\max_{i}|X_i|+ 
  tp\log\Big(\frac{en}{p}\Big)\psi   \|x\|_{\infty} \Big), 
\]
where $C$ is an absolute constant.
\end{prop}

\begin{proof} 
As in the previous proof we set $[n]=\{1,\ldots, n\}$. 
Fix $\alpha>0$ and define 
\[
   \Ga_p(\alpha)=\sup\Big\{\Big|\sum_{i=1}^n x_i X_i\Big|\colon\ x\in U_p, 
   \|x\|_{\infty}\leq \alpha \Big\}.
\]
For $E\subset [n]$ with $|E|\leq p$  let $N_E(\alpha)$ denote a $1/2$-net  
 in $\er^E\cap B_2^n\cap \alpha B_{\infty}$ with respect to the metric defined 
by  $B_2^n\cap \alpha B_{\infty}$.
We may choose $N_E(\alpha)$ of cardinality $5^{|E|}\leq 5^p$. Let 
$N(\alpha)=\bigcup_{|E|=p}N_E(\alpha)$, then 
\begin{equation}
\label{alpha_net}
  |N(\alpha)|\leq \left(\frac{5en}{p}\r)^p \quad \mbox{and}\quad
  \Ga_p(\alpha)\leq 2\sup_{x\in N(\alpha)}\Big|\sum_{i=1}^n x_i X_i\Big|.
\end{equation}

Fix $E\subset [n]$ with $|E|=p$ and $x\in N_E(\alpha)$. We have
\[
  \Big|\sum_{i=1}^n x_i X_i\Big|^2=\sum_{i=1}^nx_i^2 |X_i|^2+\sum_{i\neq j}
  \langle x_i X_i, x_j X_j\rangle.
\]
Therefore Lemma \ref{dec} gives
\[
  \Big|\sum_{i=1}^n x_i X_i\Big|^2\leq \max_{i}|X_i|^2+
  4\sup_{F\subset E}\Big|\Big\langle \sum_{i\in F}x_i X_i, 
  \sum_{j\in E\setminus F}x_j X_j\Big\rangle\Big|.
\]
Notice that for any $F\subset E$, $\max_{i\in F}|x_i|\leq \alpha$ and 
$|\sum_{j\in E\setminus F}x_j X_j|\leq  \Ga_p(\alpha)$,
hence as in the proof of Proposition \ref{estpsi1} we can show that
\[
  \Pr\Big(\Big|\Big\langle \sum_{i\in F}x_i X_i,\sum_{j\in E\setminus F}x_j X_j
  \Big\rangle\Big|> u\psi\alpha  \Ga_p(\alpha)\Big)<2^{|F|}e^{-u}.
\]
Thus
\[
  \Pr\Big(\Big|\sum_{i=1}^n x_i X_i\Big|^2>
  \max_{i}|X_i|^2+4u\psi\alpha  \Ga_p(\alpha) 
  \Big)\leq\sum_{F\subset E}2^{|F|}e^{-u}\leq 3^{|E|}e^{-u}.
\]
This together with \eqref{alpha_net} and the union bound implies
\[
  \Pr\Big( \Ga_p(\alpha)^2>4\max_{i}|X_i|^2+16u\psi\alpha
  \Ga_p(\alpha)\Big)\leq  
  \sum_{x\in N(\alpha)}3^{p}e^{-u} \leq \Big(\frac{15 en}{p}\Big)^pe^{-u}.
\]
Hence
\begin{equation} \label{est_Apalpha}
 \Pr\Big( \Ga_p(\alpha)>2\sqrt{2}\max_{i}|X_i|+32u\psi\alpha\Big)\leq 
 \Big(\frac{15 en}{p}\Big)^pe^{-u}.
\end{equation}

Using that $\Ga_p(\alpha)\geq \Ga_p(\beta)$ for $\alpha\geq \beta>0$ we 
obtain for every $\ell\geq 1$
\begin{align*}
 \Pr&\Big(\exists_{x\in U_p}\ \Big|\sum_{i=1}^n x_i X_i\Big|>
 2\sqrt{2}\max_{i}|X_i|+ u\psi\max\{\|x\|_{\infty},2^{-\ell}\}\Big)
\\
&=
 \Pr\Big(\exists_{2^{-\ell}\leq \alpha\leq 1}\  \Ga_p(\alpha)>
 2\sqrt{2}\max_{i}|X_i|+ u\psi\alpha\Big)
\\
&\leq
 \Pr\Big(\exists_{0\leq j\leq \ell-1}\ 
 \Ga_p(2^{-j})>2\sqrt{2}\max_{i}|X_i|+ \frac{1}{2}u\psi 2^{-j}\Big)
\\
&
\leq
\sum_{j=0}^{\ell-1}\Pr\Big( \Ga_p(2^{-j})>2\sqrt{2}\max_{i}|X_i|+ 
\frac{1}{2}u\psi 2^{-j}\Big)  \leq \ell \Big(\frac{15 en}{p}\Big)^pe^{-u/64},
\end{align*}
where the last inequality follows by \eqref{est_Apalpha}. 
Taking $\ell\approx\log p$ (so that $\|x\|_{\infty} \geq 2^{-\ell}|x|$) and 
$u = C t p \log (e n / p)$, we obtain the result. 
\end{proof}

\medskip

\begin{proof}[Proof of Proposition \ref{step-1}]

For any $I\subset \{1, \ldots, N\}$, the vector $P_IX$ is isotropic 
and log-concave in $\er^I$, hence it satisfies the $\psi_1$ bound with 
a universal constant.

We fix $t\geq 10$. Let $s \ge 1$ be the smallest integer such that $k
2^{-s} < k'$.    
 Set $k_\mu = \lfloor k 2^{1-\mu}\rfloor - \lfloor k 2^{-\mu}\rfloor$
 for $\mu=1, 
 \ldots, s$ and $k_{s+1} = \lfloor k2^{-s}\rfloor$.
Then 
\begin{equation} \label{eq:sum_of_k}
   \max\{1, \lfloor k 2^{-\mu}\rfloor \} \leq k_\mu \le k 2^{1-\mu}, \, \,  
    \frac{k'}{2} \leq k_{s+1} \le k', \, \,  \mbox{ and } \, \, 
   \sum_{\mu=1}^{s+1} k_\mu = k.
\end{equation}

Consider an arbitrary vector $x=(x(i))_i \in U_k$ and let
$n_1,\ldots,n_k$ be pairwise distinct integers such that $|x(n_1)|\le
|x(n_2)|\le \ldots \le |x(n_k)|$ and $x(i) = 0$ for $i\notin
\{n_1,\ldots,n_k\}$.  For $\mu = 1,\ldots,s+1$ let $F_\mu =
\{n_i\}_{j_\mu < i \le j_{\mu+1}}$, where $j_\mu = \sum_{r < \mu} k_r$
($j_1=0$). Let $x_{F_\mu}$ be the coordinate projection of $x$ onto
$\R^{F_\mu}$. Note that for each $\mu \le s$ we have $|x_{F_\mu}| \le
1$ and $\|x_{F_\mu}\|_{\infty} \le 1/\sqrt{k-j_{\mu+1}+1}\le
\sqrt{2^\mu/k}$.

The equality $x = \sum_{\mu=1}^{s+1}x_{F_\mu}$ yields that for every 
$I \subseteq\{1,\ldots,N\}$ of cardinality $m$,
\begin{align*}
  |P_I \Gamma x| \le |P_I \Gamma x_{s+1}| + \Big|\sum_{\mu=1}^s P_I
  \Gamma x_{F_\mu}\Big| \le \sup_{{I \subseteq \{1,\ldots,N\}}\atop
    {|I| = m}} \sup_{x \in U_{k'}} |P_I \Gamma x| +
  \Big|\sum_{\mu=1}^s P_I \Gamma x_{F_\mu}\Big|,
\end{align*}
where in the second inequality we used that $k_{s+1} \le k'$.

Taking the suprema over $I$ of cardinality $m$ and $x\in U_k$ we obtain
\begin{align}\label{to_prove_step_1}
  A_{k,m} = \sup_{{I \subseteq \{1,\ldots,N\}}\atop {|I| = m}} \sup_{x
    \in U_k} |P_I \Gamma x| \le \sup_{{I \subseteq
      \{1,\ldots,N\}}\atop {|I| = m}} \sup_{x \in U_{k'}} |P_I \Gamma
  x| + \sup_{{I \subseteq \{1,\ldots,N\}}\atop {|I| = m}} \sup_{x \in
    U_{k}}\Big|\sum_{\mu=1}^s P_I \Gamma x_{F_\mu}\Big|.
\end{align}
Note that
\begin{align}\label{decomposition}
\Big|\sum_{\mu=1}^s P_I \Gamma x_{F_\mu}\Big|^2 = \sum_{\mu=1}^s |P_I
\Gamma x_{F_\mu}|^2 + 2\sum_{\mu=1}^{s-1}\langle P_I \Gamma
x_{\mu},\sum_{\nu=\mu+1}^s P_I \Gamma x_{\nu}\rangle. 
\end{align}

We are going to use Proposition~\ref{estpsi2} to estimate the first 
summand and Proposition~\ref{estpsi1} to estimate the second one. 
First note that by the definition of $k'$ and $s$ we have  
$$
  {N \choose m} \leq \left(\frac{eN}{m}\r)^m\leq  \exp\left( k' \log
      \frac{en}{k'}\r)  
  \quad \mbox{ and } \quad  \frac{k}{k'} < 2^s \leq \frac{2k}{k'} \leq
  \frac{2n}{k'}.  
$$
 Hence, using the definition of $k_\mu$'s, we observe that for $t\geq
 10$ we have 
\begin{align*}
  \binom{N}{m} &\sum_{\mu=1}^s \exp\Big(- t k_\mu
  \log\Big(\frac{en}{k_\mu}\Big)\Big) \le s\, \exp\left( k' \log
    \frac{en}{k'}\r)
    \exp\Big(- t k_s \log\Big(\frac{en}{k_s}\Big)\Big)\nonumber\\
    &\le
  \frac{1}{2}\, \exp\Big(-(t k'/2) \log\Big(\frac{en}{k'}\Big)\Big). 
\end{align*}

Since $x_{F_\mu} \in U_{k_\mu}$ for every $x\in U_k$ and $\mu =
1,\ldots,s$, the union bound and Proposition \ref{estpsi2} imply that
with probability at least
\begin{displaymath}
  1 -\sum_{\mu=1}^s \binom{N}{m}\exp\Big(-t k_\mu
  \log\Big(\frac{en}{k_\mu}\Big)\Big)  
  \ge 1 - \frac{1}{2}\, \exp\Big(-(t k'/2) \log\Big(\frac{en}{k'}\Big)\Big),
\end{displaymath}
for every $x \in U_k$, every $I$ of cardinality $m$, and every $\mu
\in \{1,\ldots,s\}$  
one has 
\begin{align}\label{diagonal_part}
|P_I \Gamma x_{F_\mu}| \le C\Big(|x_{F_\mu}|\max_i|P_I X_i| + 
tk_\mu\log\Big(\frac{en}{k_\mu}\Big)\sqrt{\frac{2^\mu}{k}}\Big), 
\end{align}
where $C$ is an absolute constant. 

Similarly, by Proposition  \ref{estpsi1}, with probability at least
\begin{displaymath}
 1 - \frac{1}{2}\, \exp\Big(-(t k'/2) \log\Big(\frac{en}{k'}\Big)\Big),
\end{displaymath}
for every $x\in U_k$, every $I$ of cardinality $m$ and every $\mu\in\{1,\ldots,s\}$ 
one has 
\begin{align}\label{off_diagonal_part}
  \langle P_I \Gamma x_{\mu},\sum_{\nu =\mu+1}^s P_I \Gamma
  x_{\nu}\rangle \le Ctk_\mu
  \log\Big(\frac{en}{k_\mu}\Big)\sqrt{\frac{2^{\mu}}{k}}A_{k,m},
\end{align}
where we have used the facts that $\sum_{\nu = \mu+1}^s x_\nu \in U_{
  k_\mu}$ and  
$\sum_\nu |x_\nu|^2 \le 1$.

Using (\ref{to_prove_step_1}) -- (\ref{off_diagonal_part}) we conclude
that there  
exists an absolute constant $C_1>0$ such that with probability
at least $1 - \exp(- (t k'/2) \log(en/k'))$,
\begin{align*}
 A_{k,m}^2 &\le C\Big(\sup_{{I \subseteq \{1,\ldots,N\}}\atop {|I| = m}} 
 \sup_{x \in U_{k'}} |P_I \Gamma x|^2+\sum_{\mu=1}^s|x_{F_\mu}|^2\max_i
 \max_{{I\subseteq\{1,\ldots,N\}}\atop{|I|=m}}|P_I X_i|^2  \\
 &\phantom{aaaaa}+
 t^2\sum_{\mu=1}^s\frac{k}{2^\mu}\log^2\Big(\frac{en2^\mu}{k}\Big)  
 +t\sum_{\mu=1}^s\sqrt{\frac{k}{2^\mu}}
         \log\Big(\frac{en2^\mu}{k}\Big)A_{k,m}\Big)\\  
 &\le C_1\Big(\sup_{{I \subseteq \{1,\ldots,N\}}\atop {|I| = m}}
 \sup_{x \in U_{k'}}  
 |P_I \Gamma x|^2
 +\max_i\max_{{I\subseteq\{1,\ldots,N\}}\atop{|I|=m}}|P_I X_i|^2\\ 
 &\phantom{aaaaa}+
 t^2k\log^2\Big(\frac{en}{k}\Big)+t\sqrt{k}\log\Big(\frac{en}{k}\Big) 
 A_{k,m}\Big).
\end{align*}
Thus, with the same probability
\begin{displaymath}
  A_{k,m} \le  C_2 \Big(\sup_{{I \subseteq \{1,\ldots,N\}}\atop {|I| = m}} 
 \sup_{x \in U_{k'}} |P_I \Gamma x|+ \max_i\max_{{I\subseteq\{1,\ldots,N\}}
 \atop{|I|=m}}|P_I X_i| + t\sqrt{k}\log\Big(\frac{en}{k}\Big)\Big), 
\end{displaymath}
where $C_2>0$ is an absolute constant. 

But by Theorem~\ref{imprunifPaouris-intro}  and the union bound we
have  for every $t \ge 1$, 
\begin{align*}
 \max_i\max_{{I\subseteq\{1,\ldots,N\}}\atop{|I|=m}}|P_I X_i| \le 
 C_3 \Big(t \sqrt{m} \log(eN/m) + t\sqrt{k}\log(en/k)\Big),
\end{align*}
with probability  larger than or equal to
$$
 p_0 :=1 - n \exp\Big(-t \frac{\sqrt m \log (eN/m) +
 \sqrt{k}\log(en/k)}{\sqrt {\log em}}\Big)
$$
(we added the term depending on $k$ to get better probability, 
we may do it by adjusting $t$).
This proves the result for $t\geq 10$ with probability $p_0-\exp(- (t
k'/2) \log(en/k'))$.  
Passing to $t_0=t/20$ and adjusting absolute constants, we complete
the proof.   
\end{proof}

\subsection{Step II. $k\log\Big(\frac{en}{k}\Big) \leq
  m\log\Big(\frac{eN}{m}\Big)$, in particular $k \le k'$.}

In this case we have to be a little bit more careful than in the 
previous case with the choice of nets. We will need the following lemma, 
in which $\tilde U_k$ denotes the set of $k$-sparse vectors of the Euclidean 
norm at most one.

\begin{lemma}
\label{net_in_Uk}
Suppose that $k\leq n$,
$k_1,k_2,\ldots,k_s$ are positive integers such that $k_1+\ldots+k_s\geq k$  
and $k_{s+1}=1$. We may then find a finite subset ${\cal N}$ of 
$\frac{3}{2}\tilde U_k$ satisfying the following. \\
i) For any $x\in U_k$ there exists $y\in {\cal N}$ such that $x-y\in
\frac{1}{2} \tilde U_k$.\\ 
ii) Any $x\in{\cal N}$ may be represented in the form
$x=\pi_1(x)+\ldots+\pi_s(x)$, where 
vectors $\pi_1(x),\ldots,\pi_s(x)$ have disjoint supports,
$|\supp(\pi_i(x))|\leq k_i$ for 
$i=1,\ldots,s$,
\[
\sum_{i=1}^{s}k_{i+1}\|\pi_{i}(x)\|_{\infty}^2\leq 4
\]
and
\[
|\pi_i({\cal N})|=|\{\pi_i(x)\colon x\in{\cal N}\}|\leq
\Big(\frac{en}{k_i}\Big)^{3k_i}\quad \mbox{ for }i=1,\ldots,s.
\]
\end{lemma}

\begin{proof}

First note that we can assume that $k_1+k_2+\ldots+k_s=k$.  
Indeed, otherwise denote by $j$ the largest integer such that 
$k_j+k_{j+1}+\ldots+k_s \geq k$. If $j=s$ then set $\tilde k_j =k$, 
if $j<s$ then set $\tilde k_i = k_i$ for $j<i\leq s$, 
$\tilde k_j = k - k_{j+1}-k_{j+2}-\ldots-k_{s}$, $\pi _i(x) = 0 $ for 
$i<j$ and repeat the proof below for the sequence 
$\tilde k_j, \tilde k_{j+1}, \ldots, \tilde k_s, \tilde k_{s+1}$, where 
$\tilde k_{s+1}=1$ as before.

Recall that for $F\subset\{1,\ldots,n\}$,
$\R^F$ denotes the set of all vectors in $\er^n$ with support 
contained in $F$.

For $i=1,\ldots,s$ and $F\subset \{1,\ldots,n\}$ of cardinality at
most $k_i$ let ${\cal N}_{i}(F)$ denote the subset of
$S_F(i):=\R^F\cap B_2^n\cap k_{i+1}^{-1/2}B_{\infty}^n$ such that
\[
S_F(i)\subset {\cal N}_i(F)+\frac{k_i}{2n}\Big(B_2^n\cap
k_{i+1}^{-1/2}B_{\infty}^n\Big). 
\]
Standard volumetric argument shows that we may choose ${\cal
  N}_{i}(F)$ of cardinality at most $(6n/k_i)^{|F|}\leq
(6n/k_i)^{k_i}$ (additionally without loss of generality we assume
that $0\in {\cal N}_{i}(F)$).  We set
\[
{\cal N}_i:=\bigcup_{|F|\leq k_i} {\cal N}_{i}(F),
\]
then
\[
|{\cal N}_i|\leq \Big(\frac{en}{k_i}\Big)^{k_i}\Big(\frac{6n}{k_i}\Big)^{k_i}
\leq \Big(\frac{en}{k_i}\Big)^{3k_i}.
\]

Fix $x\in U_k$, let $F_s$ denote the set of indices of $k_s$
largest coefficients of $x$, $F_{s-1}$ -- the set of indices of the
next $k_{s-1}$ largest coefficients, etc.  Then
$x=x_{F_s}+x_{F_{s-1}}+\ldots+x_{F_1}$, $\|x_{F_s}\|_{\infty}\leq 1$
and
\[
 \|x_{F_i}\|_{\infty}\leq \frac{1}{\sqrt{k_{i+1}}} |x_{F_{i+1}}|\leq
 \frac{1}{\sqrt{k_{i+1}}}  
  \quad \mbox{ for } i<s.
\]
In particular, $x_{F_i}\in \R^{F_i}\cap B_2^n\cap
k_{i+1}^{-1/2}B_{\infty}^n$ for all $i=1,\ldots,s$. Let $\pi_i(x)$
be a vector in ${\cal N}_{i}(F_i)$ such that
\[
|x_{F_i}-\pi_i(x)|\leq \frac{k_i}{2n}\quad \mbox{ and }
\quad \|x_{F_i}-\pi_i(x)\|_{\infty}\leq \frac{k_i}{2n\sqrt{k_{i+1}}}.
\]
Define also $\pi(x)=\pi_1(x)+\ldots+\pi_s(x)$. Then
\[
|x-\pi(x)|\leq \sum_{i=1}^s|x_{F_i}-\pi_i(x)|\leq \sum_{i=1}^s
\frac{k_i}{2n}= \frac{k}{2n}\leq \frac{1}{2} 
\]
and
\begin{align*}
\sum_{i=1}^{s}k_{i+1}\|\pi_{i}(x)\|_{\infty}^2&\leq 1+
2\sum_{i=1}^{s-1}k_{i+1}(\|x_{F_i}\|_{\infty}^2
    +\|x_{F_i}-\pi_{i}(x)\|_{\infty}^2) 
\\
&\leq 1+2
\sum_{i=1}^{s-1}\Big(|x_{F_{i+1}}|^2+\Big(\frac{k_i}{2n}\Big)^2\Big)\leq  
1+2|x|^2+\frac{k^2}{2n^2}\leq 4.
\end{align*}

Thus we complete the proof by letting
\[
  {\cal N}=\{\pi(x)\colon x\in U_k\}.
\]
\end{proof}

\begin{lemma} \label{choice_of_ki}  
Suppose that $n\leq N$ and $k\leq \min\{n, k'\}$. Then for some positive 
integer $s\leq C\log\log(3m)$ we can find $s+1$ positive integers $k_1=k$, 
$k_i \in  [\frac{1}{6}m^{1/4}, m]$ for $2\leq i\leq s$, $k_{s+1}=1$ 
satisfying 
\begin{equation}\label{choiceofseq} 
 k_{i}\log\Big(\frac{en}{k_i}\Big)\leq 20g(k_{i+1}),
 \quad \mbox{ for }i=1,\ldots,s,
\end{equation}
where $C$ is an absolute positive constant and 
\[
g(z)=\left\{\begin{array}{ll}
\frac{\sqrt{zm}}{\sqrt{\log(e^2m/z)}}\log\Big(\frac{eN}{m}\Big) &
\mbox{ if } z < m,
\\
\min\Big\{\sqrt{zm}\log\Big(\frac{eN}{m}\Big),
m\log^2\Big(\frac{eN}{m}\Big)\Big\} &\mbox{ if } z\ge m .
\end{array} \right.
\]
\end{lemma}

\begin{proof}
Let us define
\[
  h(z)=z\log\Big(\frac{en}{z}\Big) \quad \mbox{ and } \quad 
  H(z)=z\log\Big(\frac{eN}{z}\Big).
\]
Notice that $h(z)\leq H(z)$, $h$ is increasing on $(0,n]$ and $H$ is increasing on $(0,N]$.
It is also easy to see that $h(\lceil z\rceil)\leq 2h(z)$ for $z\in [1,n]$.

We first establish some relations between the functions $g$ and $H$.
It is not hard to check that $\log^{3/2}(e^2m)\leq e^2\sqrt{m}$,
therefore for $z\in [1,m]$,
\begin{align}\label{g_and_H1}
  H\Big(\frac{\sqrt{zm}}{\log^{3/2}(e^2m)}\Big)&=
  \frac{\sqrt{zm}}{\log^{3/2}(e^2m)}\Big(\log\Big(\frac{eN}{m}\Big)+
  \log\Big(\frac{m\log^{3/2}(e^2m)}{\sqrt{zm}}\Big)\Big) \nonumber \\
  &\leq \frac{\sqrt{zm}}{\log^{3/2}(e^2m)}\log\Big(\frac{eN}{m}
  \Big)(1+\log(e^2m))\leq  2g(z).
\end{align}
Write $z=pm$ with $p\in (0,1)$, so 
$H(2z) = 2pm \left(\log\left(\frac{eN}{m}\r)+\log\left(\frac{m}{2 z}\r)\r)$.
Then 
\begin{equation}\label{g_and_H2}
  H(2z) \leq \frac{2 \sqrt{p}m}{\sqrt{\log(e^2/p)}}\log \Big(\frac{eN}{m}\Big)
  \sqrt{p \log(e^2/p)}(1+\log(1/p)) \leq 10 g(z),
\end{equation}
where the last inequality follows since
\begin{align*}
\sup_{p\in (0,1)}2\sqrt{p}\sqrt{\log(e^2/p)}(1+\log(1/p))
&=2\sup_{u\geq 0}e^{-u}\sqrt{2+2u}(1+2u)
\\
&\leq
2\sqrt{2}\sup_{u\geq 0}e^{-u/2}(1+2u)\leq 10.
\end{align*}

Let us define the increasing sequence $\ell_0, \ell _1, \ldots, \ell
_{s-1}$ by the formula
\[
 \ell_0=1 \quad \mbox{ and }\quad h(\ell_i)=10g(\ell_{i-1}),\ i=1,2,\ldots, 
\]
where $s$ is the smallest number such that $\ell _{s-1}\geq m$ 
(if at some moment $10 g(\ell_{j-1}) \geq n$ we set $\ell_j=m$ 
and $s=j+1$). First we show that such an $s$ exists and satisfies 
$s\leq C\log\log(3m)$ for some absolute constant $C>0$.
We will use that $h(z)\leq H(z)$. By (\ref{g_and_H1}) if $\ell _{i-1}\leq m$ 
then $\ell _i \geq \sqrt{\ell _{i-1} m }/\log ^{3/2}(e^2m)$, which implies 
\begin{equation} \label{ellodin}
  \ell_1\geq \sqrt{m}/\log^{3/2}(e^2m)\geq \frac{1}{6}m^{1/4}
\end{equation}
and, by induction, 
\[
  \ell_i\geq \Big(\frac{m}{\log^3(e^2m)}\Big)^{1-2^{-i}} \mbox{ for }
  i=0,1,2,\ldots 
\]
In particular we have for some absolute constant $C_1>0$, 
\[
  \ell_{s_1}\geq \frac{m}{2\log^3(e^2m)} \quad \mbox{ for some } s_1\leq
  C_1 \log\log(3m).
\]
 By \eqref{g_and_H2} we have $h(2z)\leq H(2z)\leq 10 g(z)$ for $z\leq m$, 
so, if $\ell _{i-1}\leq m$ then $\ell_{i}\geq 2\ell_{i-1}$. It implies that 
for some $s \leq s_1+ C_2\log\log(3m) \leq C\log\log(3m)$ we indeed have 
$\ell_{s-1}\geq m$. 
 
Finally we choose the sequence $(k_i)_{i\leq s+1}$ in the following
way.  Let $k_1=k$ and $k_i:=\min\{m,\lceil \ell_{s+1-i}\rceil\}$,
$i=2,\ldots,s+1$ (note that $k_2=m$, $k_{s+1}=1$). Using $h(\lceil
z\rceil)\leq 2h(z)$ and construction of $\ell_i$'s, we obtain
(\ref{choiceofseq}) for $i\geq 2$, while for $i=1$ by definition of
$k'$ and since $k\leq k'$ we clearly have $h(k_1)=h(k)\leq 2 g(m)= 2
g(k_2)$. Since $(\ell _i)_i$ is increasing we also observe by
(\ref{ellodin}) that $k_i\geq \frac{1}{6} m^{1/4}$ for $2\leq i\leq
s$. This completes the proof.
\end{proof}

\begin{prop}
\label{est_smallk}
Suppose that $N\geq n$ and  $k\leq \min\{n,k'\}$.
Then for $t\geq 1$,
\begin{align*}
\Pr\Big(\sup_{{I \subset \{1,\ldots,N\}}\atop{|I|=m}}
\sup_{x\in U_k}|P_I\Ga x|  &\geq
Ct\sqrt{\log\log(3m)}\sqrt{m}\log\Big(\frac{eN}{m}\Big)\Big)
\\
&\leq
\exp\Big(-\frac{t\sqrt{\log\log(3m)}\sqrt{m}\log(eN/m)}{\sqrt{\log(em)}}\Big),
\end{align*}
where $C$ is a universal constant.
\end{prop}

\begin{proof}
Let $k_1,\ldots,k_{s+1}$ be given by Lemma~\ref{choice_of_ki} and ${\cal
  N}\subset \frac{3}{2}U_k$ be
as in Lemma \ref{net_in_Uk}. Notice that
\[
  \sup_{{I \subset \{1,\ldots,N\}}\atop{|I|=m}} \sup_{x\in U_k}|P_I\Ga x| \leq 
 2\sup_{{I \subset \{1,\ldots,N\}}\atop{|I|=m}} \sup_{x\in{\cal N}}|P_I\Ga x|,
\]
so we will estimate the latter quantity.

Let us fix $x\in{\cal N}$ and $1\leq i\leq s$. We apply Theorem
\ref{singlex} to the vector
$y=\pi_i(x)/(\sqrt{k_{i+1}}\|\pi_i(x)\|_{\infty}+|\pi_i(x)|)$
(observe that $|y|\leq 1$ and $\|y\|_{\infty}\leq
1/\sqrt{k_{i+1}}$ and on the other hand 
$1/\sqrt{k_{i+1}}
\geq \frac{1}{\sqrt{m}}$) to get for $u>0$,
\begin{align*}
\Pr\Big(\sup_{{I \subset \{1,\ldots,N\}}\atop{|I|=m}}
|P_I\Ga \pi_i(x)|\geq
C(&\sqrt{k_{i+1}}\|\pi_i(x)\|_{\infty}+|\pi_i(x)|)\sqrt{m}
\log\Big(\frac{eN}{m}\Big)+u\Big)
\\
&\leq
\exp\Big(-100g(k_{i+1})\Big)
\exp\Big(-\frac{\sqrt{k_{i+1}}u}{C\sqrt{\log(em)}}\Big),
\end{align*}
where $g(x)$ is as in Lemma \ref{choice_of_ki}.

By the properties of the net ${\cal N}$ guaranteed by 
Lemma~\ref{net_in_Uk} and (\ref{choiceofseq}) 
\[
  |\pi_i({\cal N})|\exp\Big(-100g(k_{i+1})\Big) \leq 1.
\]
Therefore for all $u>0$ and $i=1,\ldots,s$,
\begin{align*}
\Pr\Big(\sup_{x\in {\cal N}}\sup_{{I \subset
    \{1,\ldots,N\}}\atop{|I|=m}}
|P_I\Ga\pi_{i}(x)|\geq
C(\sqrt{k_{i+1}}\|\pi_i(x)\|_{\infty}+&|\pi_i(x)|)\sqrt{m}
\log\Big(\frac{eN}{m}\Big)+u\Big)
\\
&\leq \exp\Big(-\frac{\sqrt{k_{i+1}}u}{C\sqrt{\log(em)}}\Big).
\end{align*}

We have for any $x\in {\cal N}$,
\begin{align*}
\sum_{i=1}^s(\sqrt{k_{i+1}}\|\pi_i(x)\|_{\infty}+|\pi_i(x)|)
&\leq
\sqrt{s}\bigg(
\Big(\sum_{i=1}^s k_{i+1}\|\pi_i(x)\|_{\infty}^2\Big)^{1/2}+
\Big(\sum_{i=1}^s|\pi_i(x)|^2\Big)^{1/2}
\bigg)
\\
&\leq \sqrt{s}\Big(2+\frac{3}{2}\Big)\leq C\sqrt{\log\log (3m)}.
\end{align*}
Therefore for any $u_1,\ldots,u_s>0$,
\begin{align*}
\Pr&\Big(\sup_{x\in {\cal N}}\sup_{{I \subset
    \{1,\ldots,N\}}\atop{|I|=m}}
|P_I\Ga x|\geq
C\sqrt{\log\log(3m)}\sqrt{m}\log\Big(\frac{eN}{m}\Big)+\sum_{i=1}^su_i\Big) 
\\
&\leq \sum_{i=1}^s\Pr\Big(\sup_{x\in {\cal
    N}}\sup_{{I \subset \{1,\ldots,N\}}\atop{|I|=m}}|P_I\Ga\pi_{i}(x)|\geq
C(\sqrt{k_{i+1}}\|\pi_i(x)\|_{\infty}
+|\pi_i(x)|)\sqrt{m}\log\Big(\frac{eN}{m}\Big)+u_i\Big)
\\
&\leq \sum_{i=1}^s \exp\Big(-\frac{\sqrt{k_{i+1}}u_i}{C\sqrt{\log
    (em)}}\Big). 
\end{align*}

Hence it is enough to choose
$u_s=Ct\sqrt{\log\log(3m)}\sqrt{m}\log(eN/m)$ and
$u_i=\frac{1}{s}Ct\sqrt{\log\log(3m)}\sqrt{m}\log(eN/m)$ for
$i=1,\ldots,s-1$ and to use the fact that $k_i\geq \frac{1}{6}m^{1/4}$
for $i=2,\ldots,s$.
\end{proof}

\subsection{Conclusion of the proof of Theorem \ref{est_akm}}

\begin{proof}
  First notice that it is sufficient to consider the case $n\leq N$.
  Indeed, if $n>N$ we may find independent isotropic $n$-dimensional
  log-concave random vectors $\tilde{X_1},\ldots,\tilde{X_n}$ such
  that $X_i=P_{\{1,\ldots,N\}}\tilde{X_i}$ for $1\leq i\leq n$.
Let $\tilde{A}$ be the $n\times n$ matrix with rows
$\tilde{X_1},\ldots,\tilde{X_n}$ and
\[
\tilde{A}_{k,m}:=\sup\{|P_{I}(\tilde{A})^*x|\colon
I\subset\{1,\ldots,n\},\ |I|=m,\ x\in U_k\}.
\]
Then obviously $\tilde{A}_{k,m}\geq A_{k,m}$ and this allows us 
to  immediately deduce the case $N\leq n$  from  the case $N=n$.

\medskip

If $\sqrt{k}\log(en/k)+\sqrt{m}\log(eN/m)\geq k'\log(en/k')$ we may
apply results of \cite{jams}.  Recall that $\Ga= A^*$.
Let $s\ge \sqrt{k}\log(en/k)+\sqrt{m}\log(eN/m)$. Applying ``in particular" 
part of Theorem~3.13 of \cite{jams} and Paouris' Theorem (inequality 
(\ref{paour_dev}) together with the union bound) to the columns of 
$m\times n$ matrix $P_I \Ga$ and adjusting corresponding constants, 
 we obtain that 
\[
 \Pr\big(\sup_{x\in U_k}|P_I\Ga x|\geq Cs\big)\leq \exp(-2s) 
\]
for any $I\subset \{1,\ldots, N\}$ with $|I|=m$ (cf. Theorem~3.6 of
\cite{jams}).  Therefore
\[
\Pr(A_{k,m}\geq Cs)\leq \sum_{|I|=m}\Pr\big(\sup_{x\in
  U_k}|P_I \Ga x|\geq Cs\big)\leq 
\binom{N}{m}\exp(-2s).
\]
By the definition of $k'$ we get
\[
\binom{N}{m}\leq \exp\big(m\log(eN/m)\big)
    \leq \exp\big(k'\log(en/k')\big),
\]
hence for $s$ as above
\[
\Pr(A_{k,m}\geq Cs)\leq \exp\big(k'\log(en/k')\big)\exp(-2s)
\leq \exp(-s)
\]
and Theorem \ref{est_akm} follows in this case.

Finally assume that  $n\leq N$ and  that
$\sqrt{k}\log(en/k)+\sqrt{m}\log(eN/m)\leq k'\log(en/k')$. 
For simplicity put  $a_k = \sqrt{k}\log(en/k)$, 
$b_m=\sqrt{m}\log(eN/m)$, $d_m= \sqrt{\log \log (3m)}$. 
If $k\leq k'$ then Theorem~\ref{est_akm} follows by 
Proposition~\ref{est_smallk} applied with 
$t _0 =t(1 + a_k/(d_m a_m))$. If $k\geq k'$ then 
we apply Proposition~\ref{est_smallk} (with the same $t_0$) and 
Propositions~\ref{step-1} with $t_1= t (b_m d_m + a_k)/(a_k +b_m)$ 
to obtain Theorem~\ref{est_akm}
(note that $C\sqrt{\frac{m}{\log(em)}}\log(eN/m)
\ge \log N\geq \log n$, so the factor $n$ in the probability in
Propositions~\ref{step-1} can be eliminated).
\end{proof}

\section{The Restricted Isometry  Property}
\label{RIP}

Fix integers $ n$ and $ N \ge 1$ and let $ A $ be an $n\times N$
matrix. Consider the problem of reconstructing any vector $x\in \R^N$
with short support (sparse vectors) from the data $ A 
x\in\mathbb{R}^n$, with a fast algorithm.

Compressive Sensing provides a way of reconstructing the original
signal $x$ from its compression $ A  x$ with $n\ll N$ by the
so-called $\ell_1$-minimization method (see \cite{D1, CRT,CT2}).

Let
$$
   \delta_m=\delta_m( A )=\sup_{x\in U_m} \left| {| A 
   x|^2}-\E | A  x|^2 \right |
$$
be the Restricted Isometry Constant (RIC) of order $m$,
introduced in \cite{CT1}. Its  important  feature
  is that if $\delta_{2m}$ is appropriately small then
every $m$-sparse vector $x$ can be reconstructed from its
compression $ A  x$ by the $\ell_1$-minimization method.
The goal is to check this property for certain models of matrices.

The articles \cite{OGproc,OGcras,jams,ALPT1,cras_alpt} considered
random matrices with independent {\em columns}, and investigated the
RIP for various models of matrices, including the log-concave Ensemble
build with independent isotropic log-concave columns.  In this
setting, the quantity $ A_{n,m}$ played a central role.

In this section we consider $n \times N$ random matrices $ A $ with
independent rows $(X_i)$.  For $T\subset\mathbb{R}^N$ the quantity
$ A_k(T)$ has been defined in (\ref{akT}) and
$ A_{k,m}= A_k(U_m)$ was estimated in the previous section.

We start with a general Lemma~\ref{dva} which 
will be used to show  that after a  
suitable discretization, 
one  can  reduce  a concentration inequality to a
deviation inequality; in particular, 
checking the RIP is reduced to estimating
$ A_{k,m}$.  It  is a slight strengthening of
Lemma~\ref{dva-intro} from the introduction.

\begin{lemma}
\label{dva}
Let $X_1, \ldots, X_n$ be independent isotropic random vectors in
$\R^N$. Let $T \subset S^{N-1}$ be a finite set.
  Let $0<\theta < 1$ and $B\geq 1$.
Then with probability at least 
$1-   |T| \exp\left( - {3 \theta ^2 n}/{8 B^2} \right)$ 
one has 
\begin{align*}
\sup_{y\in T} & \left|\frac{1}{n} 
 \sum_{i=1}^n(|\langle X_i, y\rangle|^2 - \E |\langle X_i, y\rangle|^2) 
 \right|  \\
  &\leq \theta + \frac{1 }{n} \left(  A_{k}(T) ^2 + 
  \sup _{y\in  T}\E \sum _{i=1}^n |\langle X_i, y\rangle|^2 
   \ind{|\langle X_i, y\rangle| \ge B} \right) \\
  &\leq
   \theta + \frac{1}{n}\left( A_{k}(T)^2 + \mathbb{E} 
 A_{k}(T)^2\right),   
\end{align*}
where $k\leq n$ is the largest  integer satisfying
$k\leq ( A_{k}(T)/B)^2$.
\end{lemma}

\medskip

\noindent{\bf Remark. } Note that $k$ in Lemma~\ref{dva} 
is a random variable. 

\bigskip

To prove Lemma~\ref{dva} we need 
Bernstein's inequality (see e.g., Lemma~2.2.9 in \cite{vVW}). 

\begin{prop} 
\label{Bernstein}
Let $Z_i$ be independent centered random variables such that
$|Z_i|\le a$ for all $1\le i\le n$.
Then for all $\tau \ge 0$ one has 
\begin{displaymath}
 \PP\left(\frac{1}{n} \sum_{i=1}^n Z_i \ge \tau\right) 
 \le \exp\left(-\frac{\tau^2 n}{2(\sigma^2 + a\tau/3)}\right),
\end{displaymath}
where
\begin{displaymath}
\sigma^2 = \frac{1}{n} \sum_{i=1}^n \mbox{\rm Var}(Z_i).
\end{displaymath}
\end{prop}

\noindent
{\it Proof of Lemma~\ref{dva}. } 
For $y \in T$ let
$$
  S(y) = \left|\frac{1}{n} \sum_{i=1}^n \left( |\langle
  X_i, y\rangle|^2  - \E |\langle X_i, y\rangle|^2  \r) \r|  
$$
and observe  that 
\begin{eqnarray*}
S(y) &\le&  
 \Big|\frac{1}{n} 
  \sum_{i=1}^n \left( \left(|\langle X_i, y\rangle|\wedge B \r)^2  - 
  \E \left(|\langle X_i, y\rangle|\wedge B \r)^2  \r) \Big| \\
  & + & \frac{1}{n} \sum_{i=1}^n \left( |\langle
  X_i, y\rangle|^2  - B^2 \r) \ind{|\langle X_i, y\rangle|\ge B} \\
  &+&
\frac{1}{n}\E \sum_{i=1}^n \left( |\langle 
  X_i, y\rangle|^2  - B^2 \r) \ind{|\langle X_i, y\rangle|\ge B} . 
\end{eqnarray*}
We denote  the  three summands by  $S_1(y)$, $S_2(y)$, 
$S_3(y)$, respectively,  and we estimate  each of  them separately.

\bigskip

\noindent
{\bf Estimate for $S_1(y)$:\ } 
We will use Bernstein's inequality (Proposition~\ref{Bernstein}). Given 
$y\in T$ let $Z_i (y) = \left(|\langle X_i, y\rangle|\wedge B \r)^2  - 
\E \left(|\langle X_i, y\rangle|\wedge B \r)^2$, for $i\leq n$.
Then $|Z_i(y)|\leq B^2$, so $a = B^2$. By isotropicity of $X$ for every 
$i\leq n$ one has 
$$
  \mbox{ Var}(Z_i(y))\leq \E\left( |\langle X_i, y\rangle| \wedge B\r)^4 \leq 
  \E\left( |\langle X_i, y\rangle| ^2  B^2\r) = B^2,  
$$
which implies $\sigma ^2\leq B^2$. 
By Proposition~\ref{Bernstein}, 
$$
 \PP \left( \frac{1}{n} \sum_{i=1}^n Z_i(y) \geq \theta  \r) \leq 
 \exp\left(-\frac{\theta^2 n}{2(B^2 + B^2 \theta/3) }\r) \leq 
 \exp\left(-\frac{3\theta^2 n}{8 B^2 }\r) .
$$
Then, by the union bound,
$$
  \PP\left( \sup _{y\in {T}} S_1 (y) \geq \theta \r)   = 
 \PP \left( \sup_{y\in {T}} \frac{1}{n} \ \sum_{i=1}^n Z_i(y) \geq 
 \theta  \r)  \leq 
 |T| \exp\left( - \frac{3\theta^2 n}{8 B^2} \r) .
$$

\medskip

\noindent
{\bf Estimates for $S_2(y)$ and $S_3(y)$:\ }  
For every $y\in {T}$ consider
$$
  E_B(y) = \{ i\le n \colon |\langle X_i, y\rangle|\ge B\},
$$
and let 
$$
   k'= \sup _{y\in  {T}} |E_B (y)| .
$$
Then, by the definition of $ A_{k'} (T)$, 
$$
  B^2 k' = B^2 \sup_{y\in  {T}} |E_B (y)| \leq \sup_{y\in  {T}} 
  \sum_{i \in E_B(y)}| \langle X_i, y\rangle|^2  \leq  A ^2_{k'} (T). 
$$
This yields
$$
   k'\leq \frac{ A ^2_{k'} ( {T})}{B^2} ,
$$
and therefore $k' \le k$,
where $k\leq n$ is the biggest integer satisfying 
$k\leq ( A_{k}( {T})/B)^2$.

Using the definition of $ A_{k} (T)$ again we observe 
\begin{eqnarray*}
   \sup _{y\in  {T}} S_2(y) &\le & \frac{1}{n} \ 
   \sup _{y\in  {T}} \sum _{i=1}^n   
   |\langle X_i, y\rangle|^2 \ind{|\langle X_i, y\rangle| \ge B} 
   = \frac{1}{n} \  \sup _{y\in  {T}} \sum _{i\in E_B(y)} 
   |\langle X_i, y\rangle|^2  \\
&\leq & \frac{1}{n} \ 
   \sup _{y\in  {T}} \sup _{|E| \leq k} \sum _{i\in E}   
   |\langle X_i, y\rangle|^2   \leq \frac{1}{n} \  A ^2_{k} ( {T}).
\end{eqnarray*}

 Similarly, 
$$
   \sup _{y\in  {T}} S_3(y) \le \frac{1}{n} \ 
   \sup _{y\in  {T}}\E \sum _{i=1}^n |\langle X_i, y\rangle|^2 
   \ind{|\langle X_i, y\rangle| \ge B}\leq \frac{1}{n} \ \E  A ^2_{k} ( {T}) .
$$ 
Combining estimates for $S_1(y)$, $S_2(y)$, $S_3(y)$ 
we obtain the desired result. 
\qed

\bigskip

By an approximation argument Lemma \ref{dva} has the following 
immediate consequence (cf., \cite{cras_alpt}). 

\begin{cor}  
\label{raz} 
  Let $0<\theta < 1$ and $B\geq 1$. Let $n$, $N$ be positive integers
  and $ A $ be an $n\times N$ matrix, whose rows are independent
  isotropic random vectors $X_i$, for $i\leq n$.
  Assume that  $m\leq N$ satisfies  
$$ 
   m \log \frac{11 e N}{ m} \leq \frac{3 \theta ^2 n}{16 B^2} . 
$$
Then with probability at least 
$$
   1-  \exp\left( -  \frac{3 \theta ^2 n}{16 B^2} \r) 
$$
one has 
\begin{align*}
  \delta_m\left(\frac{ A }{\sqrt n}\right)
  &=\sup_{y\in U_m} \left|\frac{1}{n} \sum_{i=1}^n(
  |\langle X_i, y\rangle|^2 - \E |\langle X_i, y\rangle|^2) \right| \\
  &\leq 2 \theta + \frac{2 }{n} \left(  A_{k,m} ^2 + 
  \sup _{y\in  {U_m}}\E \sum _{i=1}^n |\langle X_i, y\rangle|^2 
   \ind{|\langle X_i, y\rangle| \ge B} \right) \\
  &\leq
  2  \theta + \frac{2}{n}\left( A_{k,m}^2 
+ \mathbb{E}  A_{k,m}^2\right),   
\end{align*}
where $k\leq n$ is the largest integer satisfying $k\leq ( A_{k,m}/B)^2$. 
\end{cor}

\medskip

\noindent{\bf Remarks. 1. } Note that as in Lemma~\ref{dva}, 
$k$ in Corollary \ref{raz} is a random variable. \\
{\bf 2.} 
In all our applications we would like to have $ A_{k,m}^2$ and 
$\E  A_{k,m}^2$ of order $\theta n$. To obtain this, we choose 
the parameter $B$ appropriately. \\ 
{\bf 3.} Note that $ A_{k,m}$ is increasing in $k$, therefore 
we immediately have that if $m\leq N$ satisfies  
$$ 
   m \log \frac{11 e N}{\theta m} \leq \frac{3 \theta ^2 n}{16 B^2} 
  \quad \mbox{ and } \quad   \E   A_{n,m}^2 \leq \theta n  
$$
then with probability at least 
$$
  1 - \exp\left( -  \frac{3 \theta ^2 n}{16 B^2} \r) - 
  \PP\left(  A_{n,m}^2 > \theta n\r)
$$
one has  
\begin{equation} \label{eq:raz}
 \delta_m\left(\frac{ A }{\sqrt n}\right)
 \leq 6  \theta .
\end{equation}

\medskip

\noindent
{\it Proof of Corollary \ref{raz}. } 
Let $\cal{N}$ be a $1/5$-net in $U_m$ of cardinality 
${N \choose m} 11^m \leq (11 e N / m)^m$ (we can construct $\cal{N}$ in 
such a way that for every $y\in U_m$ there exists $z_y\in {\cal{N}}$ 
with such that $z_y/|z_y| \in U_m$ and $|y-z_y|\leq 1/5$). 
By the assumption on $m$,
$$
  m \log  \frac{11 e N}{ m} \leq \frac{3 \theta ^2 n}{16 B^2},
$$
and thus
$$
|{\cal N}| \exp\left( -  \frac{3 \theta ^2 n}{8 B^2} \r)
\le  \exp\left( -  \frac{3 \theta ^2 n}{16 B^2} \r).
$$
Using this and an obvious fact that $ A_k ({\cal N})
\le  A_k (U_m)$ for all $k$, we get by  Lemma \ref{dva} that 
$$
 \sup_{z\in {\cal N}} \left| \frac{1}{n} \sum_{i=1}^n(|\langle
 X_i, z\rangle|^2 - \E |\langle X_i, z\rangle|^2) \right|
 \leq  \theta + \frac{1}{n} \left( A_{k,m}^2 + 
\E  A_{k,m}^2 \right),
$$
with probability larger than or equal to $1 - \exp\left( -  \frac{3
\theta ^2 n}{16 B^2} \r)$. 

The proof is now finished by an approximation argument.  Note that
there exists a self-adjoint operator $S$ acting on the Euclidean
space $\R^N$ such that
$$
\frac{1}{n} \sum_{i=1}^n(|\langle
 X_i, z\rangle|^2 - \E |\langle X_i, z\rangle|^2)
= \langle Sz, z\rangle
$$
for all $z \in \R^N$. Now pick $w \in U_m$ such that
$$ 
|\langle Sw, w\rangle| = \sup_{y \in U_m}
|\langle Sy, y\rangle|,
$$ 
and let $I $  with $|I|=m$ contain the support of $w$. Write
$w = x +z$ where $x \in (1/5) B_2^N$ and $z \in {\cal N}$
and $x$ and $z$ are supported by $I$.  Then
\begin{eqnarray*}
  | \langle S w, w\rangle | & = &
| \langle S(x+z), (x+z)\rangle | \\
&\le & | \langle S x, x\rangle |
+ | \langle S x, z\rangle |
+ | \langle S z, x\rangle |
+ | \langle S z, z\rangle |\\
& \le & (1/25) \sup_{x \in B_2^I}
|\langle S x, x\rangle |
+ (2/5)  
\sup_{x \in B_2^I}
|\langle S x, x\rangle | 
\sup_{z\in {\cal N}} |z|
+ \sup_{z\in {\cal N}}
|\langle S z, z\rangle |.
\end{eqnarray*}
Thus
$$
\sup_{y \in U_m}
|\langle Sy, y\rangle| \le 
(25/14) \sup_{z \in {\cal N}} |\langle Sz, z\rangle|.
$$
completing the proof.
\qed

\vspace{2ex}

The following theorem  is a more general version of
Theorem  \ref{RIP-intro-thm} stated in the introduction.

\begin{thm}
\label{rip}  
Let $n$, $N$ be integers and $0<\theta < 1$. 
Let $ A $ be an $n\times N$ matrix, whose rows are independent 
isotropic log-concave random vectors $X_i$, $i\leq n$. There exists 
an absolute constant $c>0$, 
such that if $m\leq N$ satisfies  
$$
  m  \ \log\log 3m \left(\log \frac{3\max\{N, n\}}{m}\r)^2 \leq 
  \frac{c\ \theta ^2\ n}{\log (3/\theta)}  
$$
then 
$$
  \delta _m ( A /\sqrt n) \leq \theta
$$ 
with overwhelming probability.
\end{thm}

\noindent
{\bf Remark. } In fact our proof gives that there is an absolute 
constant $c>0$ such that if 
\begin{equation}\label{ripbm}
   b_m :=  m  \log \log (3m)  \left( \log \frac{3 \max\{N, n\}}{m}\r)^2 
   \le c \theta n 
\end{equation}
and 
\begin{equation}\label{ripm}
  m \log \frac{3 N}{m} \log ^2 \frac{n}{b_m} \le c \theta ^2 n 
\end{equation}
then $ \delta _m ( A /\sqrt n) \leq \theta$ with probability at least 
$$
   1-  \exp\left( - c\ \frac{ \theta ^2 n}{\log ^2 (n/b_m)} \right) - 
     2 \exp{\left(- c\ \frac{\sqrt{\log\log (3m)} \sqrt{m}}{\sqrt{\log (3m)}}\ 
  \log \frac{3 \max\{N, n\}}{m}\right)} .
$$
In particular, denoting $\alpha _n = n/\log \log (3n)$ and 
$C_{\theta} = (\theta /\log (3/\theta))^2$  one can take 
$$
   m \approx \min \left\{ \frac{\theta \alpha _n}{\log ^2 
  (\max\{N, n\} /(\theta \alpha _n))} , \ \frac{C_{\theta}  
  n}{\log (3N / (C_{\theta}  n))}  \right\}
$$
if $N\geq C_{\theta} n$ and 
$$
   m \approx \frac{ \theta  \alpha_n }{\log ^2 (\log \log (3n) / \theta)}
$$
if $N\leq C_{\theta} n$.

\proof
 Clearly it is enough to prove the estimate from the remark. So set 
$b_m$ as in (\ref{ripbm}) 
and assume that $b_m \leq c_1 \theta n$ for small enough $c_1>0$. 
Choose $B= C_1 \log \frac{n}{b_m}$, where $C_1$ is a sufficiently 
large absolute constant.

Let $k$ be as in Corollary~\ref{raz}, i.e. $k\leq n$ is the biggest integer 
satisfying $k\leq ( A_{k,m}/B)^2$. As in Theorem~\ref{est_akm} denote
$$
 \lambda_{m} = \frac{\sqrt{\log\log (3m)} \sqrt{m}}{\sqrt{\log (3m)}}\ 
  \log(e \max\{N, n\}/m)
$$
and 
$$
 \lambda_{k,m} = \sqrt{\log\log (3m)} \sqrt{m}\log(e \max\{N, n\}/m) + 
 \sqrt{k}\log(3 n/k).
$$ 
Applying (\ref{theomakm}) we obtain that there are absolute  constants 
$C_0>0$ and $c_0>0$ such that 
\begin{equation}\label{akmnorm}
     A_{k,m}  \leq C_0 \lambda_{k,m} , 
\end{equation}
with probability at least $1 - \exp{(-c_0 \lambda_{m})}$. 
By H\"older's inequality and the log-concavity assumption we also obtain
that there exists an absolute constant $C_2>0$ such that for every
$y\in {U_m}$ one has
$$
   \E \sum _{i=1}^n |\langle X_i, y\rangle|^2 \ind{|\langle X_i, y\rangle| 
   \ge B} \leq  \sum_{i=1}^n \sup_{x\in S^{n-1}}\|\langle X_i, 
   x\rangle\|_4^2\ \p\left( |\langle X_i, x\rangle| \ge B\right)^{1/2}
$$
$$
  \le n C_2 \exp(-B/C_2) \leq n C_2 (b_m/n)^{C_1/C_2}\leq c_1 \theta n
$$
for large enough $C_1$. 

Below we show that for our choice of $B$, (\ref{akmnorm}) 
implies 
\begin{equation} \label{goodk}
   \sqrt{k}\log(3 n/k)  \leq 
   \sqrt{\log\log (3m)} \sqrt{m}\log(3 \max\{N, n\}/m) =\sqrt{b_m},  
\end{equation}
 which means $ A_{k,m}  \leq 2 C_0 \sqrt{b_m}$.

Note that if $m$ satisfies (\ref{ripm}) then we can apply 
Corollary~\ref{raz} with our choice of $B$. It gives that there 
exists a positive constant $C$ such that 
$$
  \delta _m \left(\frac{ A }{\sqrt{n}}\right) \leq C \theta  
$$
with probability at least 
$$
   1-  \exp\left( -  \frac{3 \theta ^2 n}{16 B^2} \r) - 
     2 \exp{\left(- c_0 \lambda_{m}\r)},
$$
which proves the desired result.

Now we prove that (\ref{akmnorm}) implies (\ref{goodk}). 
Assume it does not hold, i.e. assume that 
$k \log ^2 \frac{3 n}{k}> b_m$. 
Then, by the definition of $k$ and (\ref{akmnorm}) 
we observe that 
$$
  k\leq \frac{ A_{k,m} ^2}{B^2} \leq \frac{C^2_0}{B^2} \lam ^2_{k,m}
  \leq 4 C_0^2 \frac{k}{B^2} \log ^2 \frac{3 n}{k}.  
$$ 
This implies that $B\leq 2 C_0 \log \frac{3 n}{k}$, which yields 
$$
    k\leq  \frac{3 n}{ \exp( B/ (2C_0) )  }. 
$$ 
Thus we obtain
\begin{align*}
  b_m &<  k \log ^2 \frac{3 n}{k} \leq 
  \frac{3 n}{\exp(B/(2C_0))}\ \frac{B^2}{4C^2_0} 
\\ 
  &= \frac{3 n}{\exp(C_1 \log (n/b_m )/(2 C_0))} \ 
 \frac{C_1^2 \log^2 (n/b_m)}{4  C_0^2} , 
\end{align*}
which is impossible for large enough $C_1$. This proves (\ref{goodk}) 
and hence completes the proof.
\qed

\section{Proofs of results from Section \ref{Section_New_Est}}
\label{proof-section}

\subsection{Proof of Theorem \ref{estN}}

\bigskip

Theorem \ref{estN} is a strengthening of the first technical result in
\cite{La}. The proof given here is a modification of the argument from
\cite{La} and we include the details for the sake of completeness.

First we show the following proposition (an analogue of Proposition~10
from \cite{La}).

\begin{prop}
\label{conditional} There exist an absolute positive constant $C_0$ 
such that the following holds. Let $X$ be an isotropic log-concave
$N$-dimensional random vector, $A=\{X\in K\}$, where $K$ is a convex set
in $\er^N$ satisfying $0<\Pr(A)\leq 1/e$. Then for every $t\geq C_0$ 
\begin{equation}
\label{cond1}
\sum_{i=1}^N\Pr(A\cap\{X(i)\geq t\})\leq 
C_0\Pr(A)\Big(t^{-2}\sigma_X^2(-\log(\Pr(A)))+Ne^{-t/C_0}\Big) 
\end{equation}
and for every $1 \leq u\leq \frac{t}{C_0}$
\begin{equation}
\label{cond2}
\left|\{i\leq N\colon \Pr(A\cap\{X(i)\geq t\})\geq e^{-u}\Pr(A)\}\r|\leq
\frac{C_0 u^2}{t^2}\sigma_X^2(-\log(\Pr(A))).
\end{equation}
\end{prop}

\begin{proof}
Let $Y$ be a random vector defined by 
\[
\Pr(Y\in B)=\frac{\Pr(A\cap\{X\in B\})}{\Pr(A)}=\frac{\Pr(X\in B\cap
  K)}{\Pr(X\in K)},
\]
i.e. $Y$ is distributed as $X$ conditioned on $A$. Clearly, for every
measurable set $B$ one has $\Pr(X\in B)\geq \Pr(A)\Pr(Y\in B)$.

It is easy to see that $Y$ is log-concave, but not necessarily isotropic. 
Without loss of generality we assume 
that $\Ex Y(1)^2\geq \Ex Y(2)^2\geq \ldots\geq \Ex Y(N)^2$ 
(otherwise we renumerate coordinates). 

Given $\alpha>0$ denote 
\[
  m=m(\alpha)=\left|\{i\colon \Ex Y(i)^2\geq \alpha\}\r|.
\]
Then $\Ex Y(1)^2\geq\ldots\geq \Ex Y(m)^2\geq \alpha$. 
Using the Paley-Zygmund inequality and log-concavity of $Y$, we get 
\begin{align*}
\Pr\Big(\sum_{i=1}^m Y(i)^2\geq \frac{1}{2}\alpha m\Big)&\geq
\Pr\Big(\sum_{i=1}^m Y(i)^2\geq \frac{1}{2}\Ex\sum_{i=1}^m Y(i)^2\Big)
\\
&\geq
\frac{1}{4}\frac{(\Ex \sum_{i=1}^m Y(i)^2)^2}{\Ex(\sum_{i=1}^m Y(i)^2)^2}\geq 
\frac{1}{C_1}.
\end{align*}
Therefore 
\[
  \Pr\Big(\sum_{i=1}^m X(i)^2\geq \frac{1}{2}\alpha m\Big)\geq 
  \Pr(A)\,  \Pr\Big(\sum_{i=1}^m Y(i)^2\geq \frac{1}{2}\alpha m\Big)\geq 
  \frac{1}{C_1}\, \Pr(A).
\]
Applying Theorem \ref{imprPaouris} (and  Chebyshev's inequality, 
cf. (\ref{paour_dev})) 
to the $m$-dimensional vector $\bar X = (X_1, \ldots, X_m)$ we observe 
\[
\Pr\Big(\sum_{i=1}^m X(i)^2\geq \frac{1}{2}\alpha m\Big)\leq
\exp\Big(-\sigma_{X}^{-1}\Big(\frac{1}{C_3}\sqrt{m\alpha}\Big)\Big)
\quad \mbox{ for }\alpha\geq C_3.
\]
Thus $\exp(-\sigma_X^{-1}(\frac{1}{C_3}\sqrt{m\alpha}))\geq \Pr(A)/C_1$ 
for $\alpha\geq C_3$, so, using the fact that $\sigma_X(tp)\leq 2 t\sigma_X(p)$ 
for $t\geq 1$, we obtain  that
\begin{equation}
\label{estm}
m(\alpha)=\left|\{i\colon \Ex Y(i)^2\geq \alpha\}\r|\leq \frac{C_4}{\alpha}
\sigma_X^2(-\log(\Pr(A)))\quad \mbox{ for }\alpha\geq C_3.
\end{equation}

Note that for every $i$ the random variable $Y(i)$ is log-concave, hence  
\[
\frac{\Pr(A\cap\{X(i)\geq t\})}{\Pr(A)}=
\Pr(Y(i)\geq t)\leq \exp\Big(1-\frac{t}{C_5 (\Ex Y(i)^2)^{1/2}}\Big).
\]
Thus, if $\Pr(Y(i)\geq t)\geq e^{-u}$ then $(\Ex Y(i)^2)^{1/2}\geq
t/(C_5 (u+1))$.  Applying \eqref{estm} with $\alpha =
t^2/(C_5(u+1))^2$ we obtain that \eqref{cond2} holds with constant
$C_6$ provided that $1\leq u\leq t/C_7$.

Now assume that $t\geq \sqrt{C_3}$ and define a nonnegative integer $k_0$ by
$2^{-k_0}t\geq \sqrt{C_3}> 2^{-k_0-1}t$. Let 
\[
I_0=\{i\colon \Ex Y(i)^2\geq t^2\},\quad I_{k_0+1}=\{i\colon \Ex
Y(i)^2<4^{-k_0}t^2\}
\]
and
\[
I_j=\{i\colon 4^{-j}t^2\leq \Ex Y(i)^2<4^{1-j}t^2\}\quad j=1,2,\ldots,k_0.
\]
Clearly $\left|I_{k_0+1}\r|\leq N$ and, by \eqref{estm}, 
\[
\left|I_j\r|\leq C_44^jt^{-2}\sigma_X^2(-\log\Pr(A)) \quad \mbox{for
  }j=0,1,\ldots,k_0.
\]
Observe also that for $j>0$ and $i\in I_j$ one has 
\[
\Pr(Y(i)\geq t)\leq \Pr\Big(\frac{Y(i)}{(\Ex Y(i)^2)^{1/2}}\geq
2^{j-1}\Big)\leq \exp\Big(1-\frac{1}{C_8}2^j\Big).
\]
Therefore 
\begin{align*}
  \sum_{i=1}^{N}\Pr(Y(i)\geq t)&=\sum_{j=0}^{k_0+1}\sum_{i\in
    I_j}\Pr(Y(i)\geq t) \leq
  \left|I_0\r|+e\sum_{j=1}^{k_0+1}\left|I_j\r|\exp\Big(-\frac{2^j}{C_8}\Big)
      \\
      &\leq
      C_4\bigg(t^{-2}\sigma_X^2(-\log\Pr(A))\bigg(1+e\sum_{j=1}^{k_0}4^{j}
      \exp\Big(-\frac{2^j}{C_8}\Big)\bigg)+eNe^{-t/(C \sqrt{C_3})}\bigg)\\
      &\leq C_1\Big(t^{-2}\sigma_X^2(-\log\Pr(A))+Ne^{-t/C_1}\Big).
\end{align*}
By the definition of $Y$, this proves (\ref{cond1}) with constant $C_1$ 
for $t\geq \sqrt{C_3}$. Taking $C_0 = \max\{C_1, \sqrt{C_3}, C_6, C_7\}$ 
completes the proof.
\end{proof}

We will use the following simple combinatorial lemma
(Lemma~11 in \cite{La}).

\begin{lemma}
\label{combf}
Let $\ell_0\geq \ell_1\geq\ldots\geq \ell_s$ be a fixed sequence of
positive integers and
\[ {\cal
  F}=\Big\{f\colon\{1,2,\ldots,\ell_0\}\rightarrow\{0,1,2,\ldots,s\}\colon\
\forall_{1\leq i\leq s}\ \left|\{r\colon f(r)\geq i\}\r|\leq
  \ell_i\Big\}.
\]
Then
\[
\left|{\cal F}\r|\leq\prod_{i=1}^s\Big(\frac{e
    \ell_{i-1}}{\ell_i}\Big)^{\ell_i}.
\]
\end{lemma}

\begin{proof}[Proof of Theorem \ref{estN}]
Since $N_X\leq N$, the statement is trivial if $t\sqrt{N}\leq C\sigma_X(p)$.
Without loss of generality we assume that
$t\sqrt{N}\geq C\sigma_X(p)$ for large enough absolute constant $C>0$. 

Let $C_0$ be the constant from Proposition \ref{conditional}. 
Since $X$ is isotropic and log-concave we may assume that  
$\Pr(X(j)\geq t)\leq e^{-t/C_0}$ for $t\geq C_0$ and $1\leq j\leq N$
(we increase the actual value of $C_0$ if needed). Fix $p\geq 1$ and 
\begin{equation}\label{condont}
 t\geq C\log \left(\frac{Nt^2}{\sigma_X^2(p)}\r).
\end{equation}
Then, for large enough $C$, $t\geq 4 C_0$ and $t^2Ne^{-t/C_0}\leq
\sigma_X^2(p)$.

Define a positive integer $\ell$ by 
\[
p\leq \ell< 2p \quad \mbox{ and }\quad \ell=2^k \mbox{ for some
  integer }k.
\]
Then $\sigma_X(p) \leq \sigma_X(\ell)\leq \sigma_X(2p)\leq 4\sigma_X(p)$. 
Since $(\Ex (N_X(t))^p)^{1/p}\leq (\Ex (N_X(t))^\ell)^{1/\ell}$,  
it is enough to show that
\[
\Ex(t^2N_X(t))^\ell\leq (C_1\sigma_X(\ell))^{2\ell}.
\]
 Define sets  
\[
B_{i_1,\ldots,i_s}=\{X(i_1)\geq t,\ldots,X(i_s)\geq t\}\quad \mbox{
  and }\quad B_{\emptyset}=\Omega
\]
and denote 
\[
 m(\ell):=\Ex N_X(t)^\ell=\Ex\Big(\sum_{i=1}^N \ind{X(i)\geq t} \Big)^\ell=
 \sum_{i_1,\ldots,i_\ell=1}^{N}\Pr(B_{i_1,\ldots,i_\ell}). 
\]
It is enough to prove 
\begin{equation}
\label{toshow1}
   m(\ell)\leq \left(\frac{C\sigma_X(\ell)}{t}\r)^{2 \ell}.
\end{equation}

We divide the sum in $m(\ell)$ into several parts.
Let $j_1\geq 2$ be an integer satisfying 
\[
   2^{j_1-2}< \log\Big(\frac{Nt^2}{\sigma_X^2(\ell)}\Big)\leq 2^{j_1-1}.
\]
Define sets 
\[
  I_{0}=\big\{(i_1,\ldots,i_\ell)\in\{1,\ldots,N\}^\ell\colon 
  \Pr(B_{i_1,\ldots,i_\ell})> e^{-\ell}\big\},
\]
\[
 I_{j}=\big\{(i_1,\ldots,i_\ell)\in\{1,\ldots,N\}^\ell\colon
 \Pr(B_{i_1,\ldots,i_\ell})\in (e^{-2^{j}\ell},e^{-2^{j-1}\ell}]
 \big\},\quad   0<j<j_1,
\]
and
\[
I_{j_1}=\big\{(i_1,\ldots,i_\ell)\in\{1,\ldots,N\}^\ell\colon
\Pr(B_{i_1,\ldots,i_\ell})\leq e^{-2^{j_1-1}\ell}\big\}.
\]
Note $\{1,\ldots,N\}^\ell=\bigcup_{j=0}^{j_1}I_j$, hence  
$m(\ell)=\sum_{j=0}^{j_1}m_j(\ell)$, where
\[
m_j(\ell):=\sum_{(i_1,\ldots,i_\ell)\in I_j}\Pr(B_{i_1,\ldots,i_\ell})\quad
\mbox{ for } 0\leq j\leq j_1.
\]

 First we estimate $m_{j_1}(\ell)$ and $m_{0}(\ell)$. Since 
$\left|I_{j_1}\r|\leq N^\ell$ 
\[
m_{j_1}(\ell)= \sum_{(i_1,\ldots,i_\ell)\in I_{j_1}}\Pr(B_{i_1,\ldots,i_\ell})\leq
N^\ell e^{-2^{j_1-1}\ell}\leq \Big(\frac{\sigma_X(\ell)}{t}\Big)^{2 \ell}.
\]

To estimate $m_{0}(\ell)$, given   
$I\subset\{1,\ldots,N\}^\ell$ and $1\leq s\leq \ell$, define 
\[
P_sI=\{(i_1,\ldots,i_s)\colon (i_1,\ldots,i_\ell)\in I \mbox{ for some }
i_{s+1},\ldots,i_\ell\}.
\]
By Proposition \ref{conditional} for $s=1,\ldots,\ell-1$ one has 
\begin{align*}
  &\sum_{(i_1,\ldots,i_{s+1})\in
    P_{s+1}I_{0}}\Pr(B_{i_1,\ldots,i_{s+1}})\leq
  \sum_{(i_1,\ldots,i_{s})\in P_{s}I_{0}}\sum_{i_{s+1}=1}^N
  \Pr(B_{i_1,\ldots,i_s}\cap\{X(i_{s+1})\geq t\})
  \\
  &\phantom{aaaaaaaaa}\leq C_0\sum_{(i_1,\ldots,i_{s})\in P_{s}I_{0}}
  \Pr(B_{i_1,\ldots,i_{s}})(t^{-2}\sigma_X^2(-\log\Pr(B_{i_1,\ldots,i_{s}}))+Ne^{-t/C_0}).
\end{align*}
Note that for $(i_1,\ldots,i_s)\in P_sI_0$ one has
$\Pr(B_{i_1,\ldots,i_s})\geq e^{-\ell}$ and, by (\ref{condont}),
$t^2Ne^{-t/C_0}\leq \sigma_X^2(p)\leq \sigma_X^2(\ell)$.  Therefore
\[
\sum_{(i_1,\ldots,i_{s+1})\in P_{s+1}I_{0}}\Pr(B_{i_1,\ldots,i_{s+1}})
\leq C_4t^{-2}\sigma_X^2(\ell)\sum_{(i_1,\ldots,i_{s})\in P_{s}I_{0}}
\Pr(B_{i_1,\ldots,i_{s}}).
\]
By  induction and since $\Pr(X(j)\geq t)\leq e^{-t/C_0}$ we obtain
\begin{align*}
m_{0}(\ell)&=\sum_{(i_1,\ldots,i_\ell)\in I_{0}}\Pr(B_{i_1,\ldots,i_\ell})\leq
(C_4t^{-2}\sigma_X^2(\ell))^{\ell-1}
\sum_{i_1\in P_1I_0}\Pr(B_{i_1})
\\
&\leq (C_4t^{-2}\sigma_X^2(\ell))^{\ell-1}Ne^{-t/C_0}\leq 
\Big(\frac{C_5\sigma_X(\ell)}{t}\Big)^{2 \ell}.
\end{align*}

Now we estimate $m_j(\ell)$ for $0<j<j_1$. 
The upper bound is based on suitable estimates for $\left|I_j\r|$. 
Fix $0<j<j_1$ and define a positive integer $r_1$ by 
\[
2^{r_1}< \frac{t}{C_0}\leq 2^{r_1+1}.
\]
For all $(i_1,\ldots,i_\ell)\in I_j$ define a function 
$f_{i_1,\ldots,i_\ell}\colon \{1,\ldots,\ell\}\rightarrow \{j,j+1,\ldots,r_1\}$
by 
\[
  f_{i_1,\ldots,i_\ell}(s)=
    \left\{
       \begin{array}{ll}
            j &\mbox{ if }\Pr(B_{i_1,\ldots,i_s})\geq 
                        \exp(-2^{j+1})\Pr(B_{i_1,\ldots,i_{s-1}}),
\\
            r &\mbox{ if } \exp(-2^{r+1})\leq 
                           \Pr(B_{i_1,\ldots,i_s})/\Pr(B_{i_1,\ldots,i_{s-1}}) 
            <\exp(-2^{r}),\ j<r<r_1,
\\
            r_1 &\mbox{ if }\Pr(B_{i_1,\ldots,i_s})< 
                             \exp(-2^{r_1})\Pr(B_{i_1,\ldots,i_{s-1}}).
\end{array}
\right.
\]
Note that for every $(i_1,\ldots,i_\ell) \in  I_j$ one has 
$$
 1= \Pr(B_{\emptyset})\geq \Pr(B_{i_1}) \geq \Pr(B_{i_1,i_2})\geq \ldots 
  \geq \Pr(B_{i_1,\ldots,i_l}) > \exp(-2^{j}\ell)
$$ 
and $f_{i_1,\ldots,i_\ell}(1)=r_1$, because 
$\Pr(X(i_1)\geq t)\leq \exp{(-t/C_0)}< \exp(-2^{r_1})
\Pr(B_{\emptyset})$.

Denote 
\[
{\cal F}_j:=\big\{f_{i_1,\ldots,i_\ell}\colon\ (i_1,\ldots,i_\ell)\in
I_j\big\}.
\]
Then for $f=f_{i_1,\ldots,i_\ell}\in {\cal F}_j$ and every $r>j$ 
one has
\[
 \exp(-2^{j}\ell)< \Pr(B_{i_1,\ldots,i_\ell}) = \prod_{s=1}^{\ell}  
 \frac{\Pr(B_{i_1,\ldots,i_s})}{\Pr(B_{i_1,\ldots,i_{s-1}})}  
 < \exp(-2^r\left|\{s\colon f(s)\geq r\}\r|).
\]
Hence for every $r\geq j$ (the case $r=j$ is trivial) one has
\begin{equation}
\label{est_l_r}
 \left|\{s\colon f(s)\geq r\}\r| \leq 2^{j-r}\ell=:\ell_r.
\end{equation}
Clearly, $\sum_{r=j+1}^{r_1}\ell_r\leq \ell$ and $\ell_{r-1}/\ell_r=2$, 
so by Lemma \ref{combf} 
\[
  \left|{\cal F}_j\r| \leq \prod_{r=j+1}^{r_1}\Big(\frac{e \ell_{r-1}}{\ell_r}
  \Big)^{\ell_r}\leq e^{2 \ell}.
\]

Now for every $f\in {\cal F}_j$ we estimate the cardinality of the set
\[
  I_j(f):=\{(i_1,\ldots,i_\ell)\in I_j\colon\ f_{i_1,\ldots,i_\ell}=f\}.
\]
Fix $f$ and for $r=j,j+1,\ldots,r_1$ set 
\[
 A_r:=\{s\in\{1,\ldots,\ell\}\colon f(s)=r\}
 \quad \mbox{ and } \quad  n_r:=\left|A_r \r|.
\]
Then $1\in A_{r_1}$ and 
\[
 n_j+n_{j+1}+\ldots+n_{r_1}=\ell . 
\]
Fixing $r<r_1$, $i_1,\ldots,i_{s-1}$, $s\in A_r$ (then $s\geq$ 2 and 
$\Pr(B_{i_1,\ldots,i_{s-1}}) \leq \Pr(B_{i_1}) \leq \exp(-t/C_0)\leq 1/e$), 
applying (\ref{cond2}) with $u=2^{r+1}\leq t/C_0$, and using the definition of 
$I_j$, we observe that $i_s$ 
may take at most
\begin{align*}
\frac{4C_0 2^{2r}}{t^2}\sigma_X^2(-\log\Pr(B_{i_1,\ldots,i_{s-1}}))
&\leq
\frac{4C_0 2^{2r}}{t^2}\sigma_X^2(2^j \ell)
\leq\frac{16 C_0 2^{2(r+j)}\sigma_X^2(\ell)}{t^2}
\\
&\leq
\frac{16 C_0 \sigma_X^2(\ell)}{t^2}\exp(2(r+j))=:m_r
\end{align*}
values in order to satisfy 
$f_{i_1,\ldots,i_\ell}=f$. 
 Thus
\[
\left|I_j(f)\r|\leq N^{n_{r_1}}\prod_{r=j}^{r_1-1}m_r^{n_r}=
N^{n_{r_1}}\Big(\frac{16 C_0 \sigma_X^2(\ell)}{t^2}\Big)^{\ell-n_{r_1}}
\exp\Big(\sum_{r=j}^{r_1-1}2(r+j)n_r\Big).
\]
Note that $\eqref{est_l_r}$ implies that 
$n_r\leq \ell_r= 2^{j-r}\ell$, so 
\[
 \sum_{r=j}^{r_1-1}2(r+j)n_r\leq 2^{j+2}\ell\sum_{r=j}^{\infty}r2^{-r} 
  = 8(j+1) \ell \leq (50 + 2^{j-2})\ell.
\]
By the definition of $r_1$ we also have
\[
 n_{r_1}\leq 2^{j-r_1}\ell\leq\frac{2 C_0}{t}2^j \ell\leq
 \frac{2^{j-3}\ell}{\log(Nt^2/(4\sigma_X^2(\ell)))},
\]
where in the last inequality we used (\ref{condont}) with large enough $C$.
Thus we obtain that for every $f\in {\cal F}_j$ 
\[
 \left|I_j(f)\r|\leq \Big(\frac{C_6\sigma_X^2(\ell)}{t^2}\Big)^\ell 
 \Big(\frac{Nt^2}{
 4\sigma_X^2(\ell)}\Big)^{n_{r_1}}\exp\big(2^{j-2}\ell\big)\leq \Big(\frac{
 C_6 \sigma_X^2(\ell)}{t^2}\Big)^\ell\exp\Big(\frac{3}{8}2^{j}\ell\Big).
\]

This implies that
\[
  \left|I_j\r|\leq \left|{\cal F}_j\r|\cdot \Big(\frac{C_6\sigma_X^2(\ell)}{t^2}
  \Big)^\ell \exp\Big(\frac{3}{8}2^{j}\ell\Big)\leq \Big(\frac{C_6 \sigma_X^2
  (\ell)}{t^2}\Big)^\ell \exp\Big(\Big(2+\frac{3}{8}2^{j}\Big)\ell\Big).
\]
Hence
\[
  m_j(\ell)=\sum_{(i_1,\ldots,i_\ell)\in I_j}\Pr(B_{i_1,\ldots,i_\ell})\leq
\left|I_j\r|\exp(-2^{j-1}\ell)\leq \Big(\frac{C_7\sigma_X^2(\ell)}{t^2}
\Big)^\ell\exp\big(-2^{j-3}\ell\big).
\]
Combining  estimates for $m_j(\ell)$'s we obtain 
\begin{align*}
m(\ell)&=m_0(\ell)+m_{j_1}(\ell)+\sum_{j=1}^{j_1-1}m_j(\ell)
\\
&\leq \Big(\frac{\sigma_X(\ell)}{t}\Big)^{2 \ell}\Big(C_5^\ell + 1 + 
\sum_{j=1}^{\infty} C_7^\ell\exp\big(-2^{j-3}\ell\big)\Big)
\leq \Big(\frac{C_8\sigma_X(\ell)}{t}\Big)^{2 \ell}, 
\end{align*}
which proves \eqref{toshow1}.
\end{proof}

\subsection{Proof of Theorem \ref{imprunifPaouris}}

Fix $t\geq 1$ and let $m_0=m_0(X,t)$.

Since $\sigma_{P_JX}\leq \sigma_X$ for every $J\subset \{1,\ldots,N\}$, 
Theorem~\ref{imprPaouris} gives for any
$J\subset \{1,\ldots,N\}$ of cardinality $m_0$,
\[
 (\Ex|P_JX|^p)^{1/p}\leq C_1(\sqrt{m_0}+\sigma_{P_JX}(p))\leq 
 C_1 (\sqrt{m}+\sigma_X(p)).
\]
Using the Chebyshev inequality and $\sigma_X(u p)\leq 2u \sigma_X(p)$ 
we observe for such $J$, 
\begin{align*}
 \Pr\Big(|P_JX|\geq 36 C_1 t\sqrt{m}\log\Big(\frac{eN}{m}\Big)\Big)&\leq
\exp\Big(-\sigma_{X}^{-1}\Big(6t\sqrt{m}\log\Big(\frac{eN}{m}\Big)\Big)\Big)
\\
&\leq
\exp\Big(-3\sigma_{X}^{-1}\Big(t\sqrt{m}\log\Big(\frac{eN}{m}\Big)\Big)\Big) . 
\end{align*}
Thus, using the definition of $m_0$ and that $\sigma_{X}^{-1} (1) =2$, we obtain 
\begin{align}
  \notag \Pr\bigg(\sup_{|J|=m_0 }|P_J X|\geq 36 C_1
  t&\sqrt{m}\log\Big(\frac{eN}{m}\Big)\bigg)
  \\
  \notag &\leq
  \binom{N}{m_0}\exp\Big(-3\sigma_{X}^{-1}\Big(t\sqrt{m}\log\Big(\frac{eN}{m}\Big)\Big)\Big)
  \\
\label{est_iu1}
&\leq
\frac{1}{2}\exp\Big(-\sigma_{X}^{-1}\Big(t\sqrt{m}\log\Big(\frac{eN}{m}\Big)\Big)\Big).
\end{align}

Now notice that
\begin{equation}
\label{est_ui2}
\sup_{|I|=m}|P_IX|=\Big(\sum_{i=1}^m|X^*(i)|^2\Big)^{1/2}\leq
\sup_{|J|=m_0}|P_JX|+\Big(\sum_{i=0}^{s-1}2^im_0|X^*(2^im_0)|^2\Big)^{1/2}
\end{equation}
with $s=\lceil \log_2(m/m_0)\rceil\leq 2 \log(em/m_0)$. By Theorem
\ref{orderstat} we get for $u\geq 0$,
\[
\Pr\Big(|X^*(2^im_0)|^2\geq C_2\log^2\Big(\frac{eN}{2^im_0}\Big)+u^2\Big)\leq
\exp\Big(-\sigma_X^{-1}\Big(\frac{1}{C_3}u2^{i/2}\sqrt{m_0}\Big)\Big).
\]
We have
\[
 C_2\sum_{i=0}^{s-1}2^im_0\log^2\Big(\frac{eN}{2^im_0}\Big)\leq 
 C_4m\log^2\Big(\frac{eN}{m}\Big).
\]
Therefore for any $u_0,\ldots,u_{s-1}\geq 0$,
\begin{align*}
\Pr\Big(\sum_{i=0}^{s-1} 2^im_0|X^{*}(2^im_0)|^2\geq C_4m\log^2\Big(\frac{eN}{m}\Big)
&+\sum_{i=0}^{s-1}u_i^2\Big)
\\
&\leq
\sum_{i=0}^{s-1}\exp\Big(-\sigma_X^{-1}\Big(\frac{1}{C_3}u_i\Big)\Big).
\end{align*}

Take $u_i^2=\frac{2 C^2_3}{s}t^2m\log^2(\frac{eN}{m})$. 
Since $s\leq 2\log(em/m_0)$, we obtain
\begin{align*}
\Pr\Big(\sum_{i=0}^{s-1} 2^im_0|X^{*}(2^i&m_0)|^2\geq
(C_4+ 2 C_3^2 t^2)m\log^2\Big(\frac{eN}{m}\Big)\Big)
\\
&\leq
 s\exp\Big(-\sigma_{X}^{-1}\Big(\frac{\sqrt{2}}{ \sqrt{s}}t \sqrt{m}
 \log\Big(\frac{eN}{m}\Big)\Big)\Big)
\\
&\leq
\frac{1}{2}
\exp\Big(-\sigma_{X}^{-1}\Big(\frac{1}{\sqrt{\log(em/m_0)}}
 t \sqrt{m}\log\Big(\frac{eN}{m}\Big)\Big)\Big).
\end{align*}
This together with \eqref{est_iu1} and \eqref{est_ui2} completes the proof.
\qed

\address


\begin{thebibliography}{99}

\bibitem{OGproc} R.~Adamczak, O.~Gu{\'e}don, A.E.~Litvak, A.~Pajor,
  and N.~Tomczak-Jaegermann, {\em Condition number of a square matrix
    with i.i.d. columns drawn from a convex body},
  Proc. Amer. Math. Soc., to appear.


\bibitem{OGcras} R.~Adamczak, O.~Gu{\'e}don, A.E.~Litvak, A.~Pajor,
  and N.~Tomczak-Jaegermann, {\em Smallest singular value of random
    matrices with independent columns}, C. R., Math.,
  Acad. Sci. Paris, {\bf 346} (2008), 853--856.



\bibitem{ALLPT-unc} R.~Adamczak, R.~Lata{\l}a, A.E.~Litvak, A.~Pajor
  and N.~Tomczak-Jaegermann, {\em Chevet type inequality and norms
    of submatrices}, preprint.



\bibitem{cras_allpt} R.~Adamczak, R.~Lata{\l}a, A.E.~Litvak, A.~Pajor
  and N.~Tomczak-Jaegermann, \emph{Geometry of log-concave Ensembles
    of random matrices and approximate reconstruction},
  C.R. Math. Acad. Sci. Paris, to appear.


\bibitem{jams} R.~Adamczak, A.E.~Litvak, A.~Pajor and
  N.~Tomczak-Jaegermann, {\em Quantitative estimates of the
    convergence of the empirical covariance matrix in log-concave
    ensembles}, J. Amer. Math. Soc. {\bf 23} (2010), 535--561.

\bibitem{ALPT1} R.~Adamczak, A.E.~Litvak, A.~Pajor and
  N.~Tomczak-Jaegermann, \emph{Restricted isometry property of
    matrices with independent columns and neighborly polytopes by
    random sampling}, Constructive Approximation, {\bf 34} (2011),
  61--88.


\bibitem{cras_alpt} R.~Adamczak, A.E.~Litvak, A.~Pajor and
  N.~Tomczak-Jaegermann, \emph{Sharp bounds on the rate of convergence
    of empirical covariance matrix}, C.R. Math. Acad. Sci. Paris, {\bf
    349} (2011), 195--200.


\bibitem{BDDW} R. Baraniuk, M. Davenport, R. DeVore, M. Wakin, \emph{A
    Simple Proof of the Restricted Isometry Property for Random
    Matrices}, Constructive Approximation, {\bf 28} (2008), 253-263.



\bibitem{Bo} C.~Borell, \emph{Convex measures on locally convex
    spaces}, Ark. Math. {\bf 12} (1974), 239--252.


\bibitem{B} J.~Bourgain, \emph{Random points in isotropic convex
    sets}, in: Convex geometric analysis, Berkeley, CA, 1996,
  Math. Sci. Res. Inst. Publ., Vol. 34, 53--58, Cambridge Univ. Press,
  Cambridge (1999).




\bibitem{CRT} E.J. Candes, J. Romberg and T. Tao, \emph{ Stable signal
    recovery from incomplete and inaccurate measurements}, Comm. Pure
  App. Math.  {\bf 59}, (2006), 1207--1223.


\bibitem{CT1} E.J.~Cand\'es and T.~Tao, \emph{Decoding by linear
    programming}, IEEE Trans. Inform. Theory {\bf 51} (2005),
  {4203--4215}.

\bibitem{CT2} E. J. Candes and T. Tao, \emph{Near-optimal signal
    recovery from random pro- jections: universal encoding
    strategies}, IEEE Trans. Inform. Theory, {\bf 52} (2006),
  5406--5425.





\bibitem{DKH} Ju.S.~Davidovic, B.I.~Korenbljum and B.I.~Hacet, \emph{A
    certain property of logarithmically concave functions}, Soviet
  Math. Dokl. {\bf 10} (1969), 447--480; translation from
  Dokl. Akad. Nauk SSSR {\bf 185} (1969), 1215--1218.


\bibitem{D1}{D.L.~Donoho}, \emph{Neighborly Polytopes and Sparse
    solutions of underdetermined linear equations}, {Department of
    {S}tatistics, {S}tanford {U}niversity}, {2005}.


\bibitem{GK} E.D.~Gluskin and S.~Kwapie\'n, \emph{Tail and moment
    estimates for sums of independent random variables with
    logarithmically concave tails}, Studia Math.  {\bf 114} (1995)
  303--309.

\bibitem{KLS} R.~Kannan, L.~Lov\'{a}sz and M. Simonovits, \emph{Random
    walks and $O\sp *(n\sp 5)$ volume algorithm for convex bodies},
  {Random structures and algorithms}, {\bf 2} (1997), 1--50.



\bibitem{La} R.~Lata{\l}a, \emph{Order statistics and concentration of
    $l_r$ norms for log-concave vectors}, {J. Funct. Anal.} {\bf 261}
  (2011), 681--696.


\bibitem{La1} R.~Lata{\l}a, \emph{Weak and strong moments of random
    vectors}, preprint, http://arxiv.org/abs/1012.2703.

\bibitem{LaWo} R.~Lata{\l}a and J.O. Wojtaszczyk, \emph{On the infimum
    convolution inequality}, {Studia Math.} {\bf 189} (2008),
  147--187.

\bibitem{LT} M.~Ledoux and M.~Talagrand, \emph{Probability in Banach
    spaces. Isoperimetry and processes}, Springer-Verlag, Berlin,
  1991.


\bibitem{M1}{S.~Mendelson}, \emph{Empirical Processes with a bounded
    $\psi_1$ diameter}, {Geom. Funct. Anal.}, {\bf 20} (2010),
  {988--1027}.

\bibitem{MPT}{S.~Mendelson, A.~Pajor and N.~Tomczak-Jaegermann},
  \emph{Reconstruction and subgaussian operators in asymptotic
    geometric analysis}, {Geom. Funct. Anal.} {\bf 17} (2007),
  {1248--1282}.


\bibitem{Pa} G.~Paouris, \emph{Concentration of mass on convex
    bodies}, {Geom. Funct. Anal.} {\bf 16} (2006), 1021--1049.

\bibitem{Pr} A. Pr{\'e}kopa, \emph{Logarithmic concave measures with
    application to stochastic programming}, {Acta Sci. Math.}  {\bf
    32} (1971), 301--316.

\bibitem{R} M. Rudelson, {\em Random vectors in the isotropic
    position}, J. Funct. Anal.  {\bf 164} (1999), 60--72.


\bibitem{SV} N. Srivastava and R. Vershynin, \emph{Covariance Estimation
    for distributions with $2 + \varepsilon$ moments}, preprint.


\bibitem{vVW} A.W. van der Vaart and J.A. Wellner, \emph{Weak
    convergence and empirical processes. With applications to
    statistics,} Springer Series in Statistics, Springer-Verlag, New
  York, 1996.

\bibitem{V} R. Vershynin, \emph{How close is the sample covariance
    matrix to the actual  covariance
    matrix?}, preprint.


\end{thebibliography}
\end{document}